\documentclass[12pt]{amsart}

\usepackage{amsmath, amssymb, amsfonts, amscd, pictexwd, dcpic}

\newtheorem{theorem}{Theorem}[section]
\newtheorem{lemma}[theorem]{Lemma}

\newtheorem{definition}[theorem]{Definition}
\newtheorem{corollary}[theorem]{Corollary}
\newtheorem{remark}[theorem]{Remark}
\newtheorem{example}[theorem]{Example}

\newtheorem{note}[theorem]{Note}

\newcommand{\from}{\colon}

\usepackage{tikz}
\begin{document}
\title[Zero-Divisor Graphs of $\mathbb{Z}_n$,  their products and $D_n$]{Zero-Divisor Graphs of $\mathbb{Z}_n$,  their products and $D_n$}

\author{Amrita Acharyya}
\address{Department of Mathematics and Statistics\\
University of Toledo,  Main Campus\\
Toledo,  OH 43606-3390}
\email{Amrita.Acharyya@utoledo.edu}

\author{Robinson Czajkowski}
\address{Department of Mathematics and Statistics\\
University of Toledo,  Main Campus\\
Toledo,  OH 43606-3390}
\email{ Robinson.Czajkowski@rockets.utoledo.edu}



\begin{abstract}
This paper is an endeavor to discuss some properties of zero-divisor graphs of the ring $\mathbb{Z}_n$,  the ring of integers modulo $n$. The zero divisor graph of a commutative ring $R$,  is an undirected graph whose vertices are the nonzero zero-divisors of $R$,  where two distinct vertices are adjacent if their product is zero. The zero divisor graph of $R$ is denoted by $\Gamma(R)$. We discussed $\Gamma(\mathbb{Z}_n)$'s by the attributes of completeness,  k-partite structure,  complete k-partite structure,  regularity,  chordality,  $\gamma - \beta$ perfectness,  simplicial vertices. The clique number for arbitrary $\Gamma(\mathbb{Z}_n)$ was also found. This work also explores related attributes of finite products $\Gamma(\mathbb{Z}_{n_1}\times\cdots\times\mathbb{Z}_{n_k})$,  seeking to extend certain results to the product rings. We find all $\Gamma(\mathbb{Z}_{n_1}\times\cdots\times\mathbb{Z}_{n_k})$ that are perfect. Likewise,  a lower bound of clique number of $\Gamma(\mathbb{Z}_m\times\mathbb{Z}_n)$ was found. Later,  in this paper we discuss some properties of the zero divisor graph of the poset $D_n$,  the set of positive divisors of a positive integer $n$ partially ordered by divisibility.
\end{abstract}

\maketitle

%%%%%%%%%%%%%%  Introduction

\section{Introduction} \label{s:Intro}

Zero-divisor graphs were first discussed by Beck~\cite{nB66} as a way to color commutative rings. They were further discussed by Livingston and Anderson in ~\cite{jS83} and ~\cite{jW98}. A zero-divisor graph of a ring $R$,  denoted by $\Gamma(R)$,  is a graph whose vertices are all the zero divisors of $R$. Two distinct vertices $u$ and $v$ are adjacent if $uv = 0$. Beck~\cite{nB66} considered every element of $R$ a vertex,  with 0 sharing an edge with all other vertices. Since then,  others have chosen to omit 0 from zero-divisor graphs [2,  3,  4,  5]. For our purposes,  we omit 0 so that the vertex set of $\Gamma(\mathbb{Z}_n)$ denoted by $ZD(\mathbb{Z}_n)$ will only be the non-zero zero-divisors.\\
In the first section,  we explore a concept explored by Smith ~\cite{jK55} called type graphs. In ~\cite{jK55},  type graphs were used to find all perfect $\Gamma(\mathbb{Z}_n)$. We extended the notion of type graphs for $\Gamma(\mathbb{Z}_{n_1}\times\cdots\times\mathbb{Z}_{n_k})$ to find all perfect zero-divisor graphs of such products, where $n_1, n_2, \cdots , n_k$ are positive integers and $\Gamma(\mathbb{Z}_{n_1}\times\cdots\times\mathbb{Z}_{n_k})$ is the direct product of $Z_{n_i}$s, $1\le i\le k$. We then move on to various properties of $\Gamma(\mathbb{Z}_n)$ and $\Gamma(\mathbb{Z}_{n_1}\times\cdots\times\mathbb{Z}_{n_k})$. AbdAlJawad and Al-Ezeh ~\cite{bH77} discussed the domination number of $\Gamma(\mathbb{Z}_n)$. We extend this result to find an upper bound and lower bound for the domination number of finite product $\Gamma(\mathbb{Z}_{n_1}\times\cdots\times\mathbb{Z}_{n_k})$ and discussed coefficient of smallest degree of domination polynomial of $\Gamma(\mathbb{Z}_n)$. In the last section,  we explore zero divisor graphs of the poset $D_n$,  the set of positive divisors of a positive integer $n$ partially ordered by divisibility and we catalog them in a similar way. Zero divisor graph of poset is studied in \cite{jW99}, \cite{jW100}, \cite{jW101}.

%%%%%%%%%%%%%%  Preliminaries

\section{Type Graphs}

When we consider zero-divisor graphs of $\Gamma(\mathbb{Z}_n)$,  it is useful to consider the type graphs of these rings. A type graph has vertices of $T_a$ where $a$ is a factor of $n$ that is neither 1 nor 0. The set of all $T_i$ forms a partition of the zero divisor graph by $T_a = \{ x \in ZD(\mathbb{Z}_n) | gcf(x,  n) = a \}$. This concept was shown by Smith ~\cite{jK55},  where the type graph was used to find all perfect $\Gamma(\mathbb{Z}_n)$. Smith used the notation $\Gamma^T(\mathbb{Z}_n)$ to denote the type graph. In that paper,  four key observations were shown to be true regarding the type graphs on $\mathbb{Z}_n$. In this section,  we modify the definition of type graph to fit the graph of $\mathbb{Z}_{n_1}\times\mathbb{Z}_{n_2}\times\cdots\times\mathbb{Z}_{n_k}$. Additionally,  we show these observations to be true over this type graph as well. We then use analogues of some theorems from ~\cite{jK55} to characterize the perfectness of $\Gamma(\mathbb{Z}_{n_1}\times\mathbb{Z}_{n_2}\times\cdots\times\mathbb{Z}_{n_k})$.\\

The following are two important theorems from ~\cite{jK55}.

\begin{theorem}[Smith's Main Theore] ~\cite{jK55} 
	A graph $\Gamma(\mathbb{Z}_n)$ is perfect iff $n$ is of one of the following forms:
	\begin{enumerate}	
		\item[1.] $n = p^a$ for prime $p$ and positive integer $a$.
		\item[2.] $n = p^aq^b$ for distinct primes $p, q$ and positive integers $a, b$.
		\item[3.] $n = p^aqr$ for distinct primes $p, q, r$ and positive integer $a$.
		\item[4.] $n = pqrs$ for distinct primes $p, q, r, s$.
	\end{enumerate}
\end{theorem}

\begin{theorem}[Simth's Theorem 4.1] ~\cite{jK55}
$\Gamma(\mathbb{Z}_n)$ is perfect iff its type graph $\Gamma^T(\mathbb{Z}_n)$ is perfect.\\
\end{theorem}

\begin{definition}[Type graph of $\mathbb{Z}_{n_1}\times\mathbb{Z}_{n_2}\times\cdots\times\mathbb{Z}_{n_k}$]

The type graph of $\mathbb{Z}_{n_1}\times\mathbb{Z}_{n_2}\times\cdots\times\mathbb{Z}_{n_k}$ denoted by $\Gamma^T(\mathbb{Z}_{n_1}\times\mathbb{Z}_{n_2}\times\cdots\times\mathbb{Z}_{n_k})$ has a vertex set of the type classes $T(x_1,  x_2, \cdots, x_k)$ where $(x_1,  x_2, \cdots, x_k) \neq (0, 0, \cdots, 0)$ nor $(1, 1, \cdots, 1)$,  and $x_i$ is a divisor of $n_i$,  1,  or 0.\\
$T(x_1, x_2, \cdots, x_k) = \{ (a_1, a_2, \cdots, a_k) \mid| a_i \in \mathbb{Z}_{n_i}/0$ and $gcf(a_i,  n_i) = x_i$ or $a_i=0$ if $x_i=0$ $\}$. Arbitrary $T(x_1, x_2, \cdots, x_k)$ shares an edge with arbitrary $T(y_1, y_2, \cdots, y_k)$ iff $x_iy_i = 0$ for all $i$.

\end{definition}
	
Smith ~\cite{jK55} gave the following four observations for the type graph of $\Gamma(\mathbb{Z}_n)$. 

\begin{theorem}\label{1.1} Each vertex of $\Gamma(\mathbb{Z}_{n})$ is in exactly one type class.
\end{theorem}

\begin{theorem}\label{1.2} Arbitrary distinct vertices $T_x$ and $T_y$ share an edge in $\Gamma^T(\mathbb{Z}_{n})$ iff each $a\in T_x$ shares an edge with each $b\in T_y$ in $\Gamma(\mathbb{Z}_{n})$.
\end{theorem}

\begin{theorem}\label{1.3} Arbitrary distinct vertices $T_x $ and $T_y $ don't share an edge in $\Gamma^T(\mathbb{Z}_{n})$ iff each $a\in T_x$ doesn't share an edge with each $b\in T_y$ in $\Gamma(\mathbb{Z}_{n})$.
\end{theorem}

\begin{theorem}\label{1.4} In $\Gamma(\mathbb{Z}_{n})$ consider arbitrary $a$ and $b$ in the same type class. An arbitrary vertex $c$ in $\Gamma(\mathbb{Z}_{n})$ shares an edge with $b$ iff it shares an edge with $a$ also.
\end{theorem}

Following are the four analogues to the above results for  $\Gamma^T(\mathbb{Z}_{n_1}\times\mathbb{Z}_{n_2}\times\cdots\times\mathbb{Z}_{n_k})$.\\

\begin{theorem}\label{1.1 a} Each vertex of $\Gamma(\mathbb{Z}_{n_1}\times\mathbb{Z}_{n_2}\times\cdots\times\mathbb{Z}_{n_k})$ is in exactly one type class.
\end{theorem}
\begin{proof}
Assume otherwise. Then there is a vertex $v$ that is not in any type class,  or $v$ is in multiple type classes.\\
\begin{enumerate}
\item[Case 1:] $v$ is not in any type class.\\

Then $v$ must have an element $a_i$ that is not 0 and whose gcf with $n_i$ is not a number $x_i$ which is clearly not true.\\
\item[Case 2:] $v$ is in multiple type classes.\\

Let $v = (a_1,  a_2,  \cdots,  a_k) \in T(x_1,  x_2, \cdots,  x_k) \cap T(y_z,  y_2,  \cdots,  y_k)$. Then for all $i \in \{1,  2,  \cdots,  k\}$ if $a_i = 0$,  then $x_i = y_i = 0$ and if $a_i \neq 0$,  then $gcd(a_i,  n_i) = x_i = y_i$ giving $(x_1,  x_2, \cdots,  x_k) = (y_1,  y_2, \cdots,  y_k)$ which is a contradiction.
\end{enumerate}
\end{proof}

\begin{theorem}\label{1.2a} Arbitrary distinct vertices $T_x = T(x_1, x_2, \cdots, x_k)$ and $T_y = T(y_1, y_2, \cdots, y_k)$ share an edge in $\Gamma^T(\mathbb{Z}_{n_1}\times\mathbb{Z}_{n_2}\times\cdots\times\mathbb{Z}_{n_k})$ iff each $a\in T_x$ shares an edge with each $b\in T_y$ in $\Gamma(\mathbb{Z}_{n_1}\times\mathbb{Z}_{n_2}\times\cdots\times\mathbb{Z}_{n_k})$.
\end{theorem}

\begin{proof}
Let $T_x$ shares and edge with $T_y$. By the definition,  $x_iy_i = 0$ for every $i$. Consider arbitrary $(a_1, \cdots, a_i)\in T_x$ and $(b_1, \cdots, b_i)\in T_y$.Since each $a_i$ is a multiple of $x_i$ and each $b_i$ is a multiple of $y_i$,  $a_ib_i$ is a multiple of $x_iy_i$ and therefore equal to 0. Then $(a_1, \cdots, a_i)$ and $(b_1, \cdots, b_i)$ share an edge.\\
Conversely,  let every $a\in T_x$ and $b\in T_y$ share an edge. Since $x = (x_1, x_2, \cdots, x_k)$ is an element of $T_x$,  and $y = (y_1, y_2, \cdots, y_k)$ is an element of $T_y$,  $x$ and $y$ share an edge. Then $T_x$ must share an edge with $T_y$.
\end{proof}

\begin{theorem}\label{1.3a} Arbitrary distinct vertices $T_x = T(x_1, x_2, \cdots, x_k)$ and $T_y = T(y_1, y_2, \cdots, y_k)$ don't share an edge in $\Gamma^T(\mathbb{Z}_{n_1}\times\mathbb{Z}_{n_2}\times\cdots\times\mathbb{Z}_{n_k})$ iff each $a\in T_x$ doesn't share an edge with each $b\in T_y$ in $\Gamma(\mathbb{Z}_{n_1}\times\mathbb{Z}_{n_2}\times\cdots\times\mathbb{Z}_{n_k})$.
\end{theorem}

\begin{proof}
Let $T_x$ does not share an edge with $T_y$. By the definition,  $x_iy_i \neq 0$ for some $i$,  which means $x_iy_i$ lacks some factor $f$ of $n_i$. Consider arbitrary $(a_1, \cdots, a_i\cdots, a_k)\in T_x$ and $(b_1, \cdots, b_i,  \cdots b_k)\in T_y$. Now, $a_i$ is a multiple of $x_i$ and  $b_i$ is a multiple of $y_i$,  and thus,  $a_ib_i$ is a multiple of $x_iy_i$. Since $gcf(a_i, n_i) = x_i$ and $gcf(b_i, n_i) = y_i$,  $a_ib_i$ also lacks the factor $f$ from $n_i$ and is therefore non-zero. So $(a_1, \cdots, a_k)$ and $(b_1, \cdots, b_k)$ do not share an edge.\\
Conversely,  let each $a\in T_x$ and $b\in T_y$ do not share an edge. Since $x = (x_1, x_2, \cdots, x_k)$ is an element of $T_x$,  and $y = (y_1, y_2, \cdots, y_k)$ is an element of $T_y$,  $x$ and $y$ don't share an edge. Then $T_x$ must not share an edge with $T_y$.
\end{proof}

\begin{theorem}\label{1.4a} In $\Gamma(\mathbb{Z}_{n_1}\times\mathbb{Z}_{n_2}\times\cdots\times\mathbb{Z}_{n_k})$ consider arbitrary $a = (a_1, a_2, \cdots, a_k)$ and $b = (b_1, b_2, \cdots, b_k)$ in the same type class $T(t_1, t_2, \cdots, t_k)$. An arbitrary vertex $c = (c_1, c_2, \cdots, c_k)$ shares an edge with $b$ iff it shares an edge with $a$ also.
\end{theorem}
\begin{proof}
Follows from Theorem \ref{1.2} and \ref{1.3}.
\end{proof}

Next,  we want have the following theorem:\

\begin{theorem}\label{1.8} $\Gamma(\mathbb{Z}_{n_1}\times\mathbb{Z}_{n_2}\times\cdots\times\mathbb{Z}_{n_k})$ is perfect iff its type graph $\Gamma^T(\mathbb{Z}_{n_1}\times\mathbb{Z}_{n_2}\times\cdots\times\mathbb{Z}_{n_k})$ is perfect.
\end{theorem}

To show this,  we will use the following three theorems,  whose proofs are analogous to the corresponding proofs in ~\cite{jK55}.\\

\begin{theorem}\label{1.5} Given arbitrary hole or antihole $H$ of length greater than $4$ in $\Gamma(\mathbb{Z}_{n_1}\times\mathbb{Z}_{n_2}\times\cdots\times\mathbb{Z}_{n_k})$,  every vertex in $H$ belongs to a different type class.
\end{theorem}

\begin{theorem}\label{1.6} Let there be a hole or antihole $H$ length $l>4$ in $\Gamma(\mathbb{Z}_{n_1}\times\mathbb{Z}_{n_2}\times\cdots\times\mathbb{Z}_{n_k})$. Then the type graph $\Gamma^T(\mathbb{Z}_{n_1}\times\mathbb{Z}_{n_2}\times\cdots\times\mathbb{Z}_{n_k})$ must also contain a hole or antihole length $l$.
\end{theorem}

\begin{theorem}\label{1.7} 
	Let there be a hole or antihole $H$ length $l>4$ in the type class $\Gamma^T(\mathbb{Z}_{n_1}\times\mathbb{Z}_{n_2}\times\cdots\times\mathbb{Z}_{n_k})$. Then the graph $\Gamma(\mathbb{Z}_{n_1}\times\mathbb{Z}_{n_2}\times\cdots\times\mathbb{Z}_{n_k})$ must also contain a hole or antihole length $l$.
\end{theorem}

Using these theorems,  now we can establish the following proof of Theorem \ref{1.8}.\\
\begin{proof}
The proof is analogous to the proof in ~\cite{jK55}.
\end{proof}

Now that we know perfectness in the type graph implies perfectness in the zero-divisor graph,  it is possible to find all such perfect $\Gamma(\mathbb{Z}_{n_1}\times\mathbb{Z}_{n_2}\times\cdots\times\mathbb{Z}_{n_k})$. As it turns out,  for both $\Gamma^T(\mathbb{Z}_n)$ and $\Gamma^T(\mathbb{Z}_{n_1}\times\mathbb{Z}_{n_2}\times\cdots\times\mathbb{Z}_{n_k})$,  we can exchange the primes of each $n_i$,  and as long as the form of the primes (the amount of primes and the power of each prime) stays the same,  the type graph will be isomorphic. To illustrate this,  consider $\Gamma^T(\mathbb{Z}_{p^2q}\times\mathbb{Z}_p)$ where $p, q$ are prime. This type graph is isomorphic to $\Gamma^T(\mathbb{Z}_{r^2s}\times\mathbb{Z}_t)$ where $r, s, t$ are prime,  even if the value of the primes change. We will use this to find all perfect $\Gamma(\mathbb{Z}_{n_1}\times\mathbb{Z}_{n_2}\times\cdots\times\mathbb{Z}_{n_k})$.

\begin{theorem}\label{1.9} Consider some $\Gamma^T(\mathbb{Z}_n)$ and $\Gamma^T(\mathbb{Z}_m)$ such that $n=p_1^{\alpha_1}p_2^{\alpha_2}\cdots p_k^{\alpha_k}$ and $m=q_1^{\alpha_1}q_2^{\alpha_2}\cdots q_k^{\alpha_k}$. Then $\Gamma^T(\mathbb{Z}_n) \cong \Gamma^T(\mathbb{Z}_m)$.
\end{theorem}
\begin{proof}
Consider arbitrary vertex $u$ in $\Gamma^T(\mathbb{Z}_n)$. $u$ is a factor of $n$,  so we can write $u=p_1^{x_1}p_2^{x_2}\cdots p_k^{x_k}$. Note that $0\leq x_i \leq \alpha_i$,  $\forall i$. Define a function $f:\Gamma^T(\mathbb{Z}_n)\to \Gamma^T(\mathbb{Z}_m)$ as $f(u)=f(p_1^{x_1}p_2^{x_2}\cdots p_k^{x_k}) = q_1^{x_1}q_2^{x_2}\cdots q_k^{x_k}$. Since $n$ and $m$ both have the same amount of prime factors,  and each corresponding prime has the same power $\alpha_i$,  the result follows.\\
 \end{proof}

\begin{theorem}\label{1.10} Consider $\Gamma^T(\mathbb{Z}_{n_1}\times\cdots\times\mathbb{Z}_{n_k})$ and $\Gamma^T(\mathbb{Z}_{m_1}\times\cdots\times\mathbb{Z}_{m_k})$ where the prime factorization of $n_i$ has the same form as $m_i$ for each $i$. That is,  $n_i$ and $m_i$ have the same amount of prime factors and the same power for each prime. Then $\Gamma^T(\mathbb{Z}_{n_1}\times\cdots\times\mathbb{Z}_{n_k}) \cong \Gamma^T(\mathbb{Z}_{m_1}\times\cdots\times\mathbb{Z}_{m_k})$.
\end{theorem}
\begin{proof}
Take arbitrary $n_i$.\\
Denote the prime factorization of $n_i = p_{i,  1}^{\alpha_{i,  1}}\cdots p_{i,  j_i}^{\alpha{i,  j_i}}$ where $j_i$ is the amount of prime of $n_i$. Likewise,  $m_i = q_{i,  1}^{\alpha_{i,  1}}\cdots q_{i,  j_i}^{\alpha{i,  j_i}}$. Note that the only difference between these factorizations are the value of the primes being used. The powers and amount of primes are the same. Consider arbitrary $(u_1, \cdots, u_k) \in \Gamma^T(\mathbb{Z}_{n_1}\times\cdots\times\mathbb{Z}_{n_k})$. Each $u_i$ is a factor of $n_i$ or 0. We can write $u_i = p_{i,  1}^{x_{i,  1}}\cdots p_{i,  j_i}^{x_{i,  j_i}}$ where $0 \leq x_{i,  l} \leq \alpha_{i,  l}$. Note that if $u_i$ is 1,  each $x_{i,  l}$ is 0 and if $u_i$ is 0,  $x_{i,  l} = \alpha_{i,  l}$ for every $l$.\\
Define a function $f: \Gamma^T(\mathbb{Z}_{n_1}\times\cdots\times\mathbb{Z}_{n_k}) \to \Gamma^T(\mathbb{Z}_{m_1}\times\cdots\times\mathbb{Z}_{m_k})$ as $f(u_1, \cdots, u_k) = f(p_{1,  1}^{x_{1,  1}}\cdots p_{1,  j_1}^{x_{1,  j_1}}, \cdots, p_{k,  1}^{x_{k,  1}}\cdots p_{k,  j_k}^{x_{k,  j_k}})$\\
$= (q_{1,  1}^{x_{1,  1}}\cdots q_{1,  j_1}^{x_{1,  j_1}}, \cdots, q_{k,  1}^{x_{k,  1}}\cdots q_{k,  j_k}^{x_{k,  j_k}}) = (v_1, \cdots, v_k)$. Note that all we did was only replaced the primes. Hence the result follows as the previous one.
\end{proof}

\begin{theorem}\label{1.11} $\Gamma^T(\mathbb{Z}_{n_1}\times\cdots\times\mathbb{Z}_{n_k})$ is isomorphic to $\Gamma^T(\mathbb{Z}_{n_1\cdots n_k})$ if all $n_i$'s are mutually co-prime.
\end{theorem}
\begin{proof}
The proof follows by Chineese Remainder theorem.
\end{proof}

The next theorem will show how we can characterize perfectness of $\Gamma(\mathbb{Z}_{n_1}\times\cdots\times\mathbb{Z}_{n_k})$. Because now by the above three theorem without loss of generality we can simply choose primes that will make each $n_i$ co-prime. Then we know the type graph will be isomorphic to $\Gamma(\mathbb{Z}_n)$ where $n$ is the product of all such co-prime $n_i$. So $n$ will have a prime factorization with the total amount of primes in all $n_i$ and they will have corresponding powers. So,  we have the following theorem.

\begin{theorem}\label{1.12} $\Gamma(\mathbb{Z}_{n_{1}}\times\mathbb{Z}_{n_{2}}\cdots \mathbb{Z}_{n_{k}})$ is perfect iff it is possible to find mutually co prime positive integers $m_1, m_2 \cdots m_k$,  so that each $m_{i}$ has same amount of prime factors with same exponent in it's prime factorization as that in $n_i$ and $\Gamma(\mathbb{Z}_{m_{1}m_{2}\cdots m_{k}})$ is perfect. 
\end{theorem}

\begin{example}
For example,  $\Gamma(\mathbb{Z}_{p^2q}\times\mathbb{Z}_{p})$ is perfect because $\Gamma(\mathbb{Z}_{a^2bc})$ is perfect as shown by ~\cite{jK55}. Also note,  no product with a dimension greater than four can be perfect. $\Gamma(\mathbb{Z}_{p_1}\times\cdots\times\mathbb{Z}_{p_5})$ is not perfect since no $\Gamma(\mathbb{Z}_{p_1\cdots p_5})$ is perfect as shown by ~\cite{jK55}.
\end{example}

\section{Some properties of $\Gamma(\mathbb{Z}_n)$}

In this section we characterize $\Gamma(\mathbb{Z}_n)$ by various qualities such as completeness,  cordiality and clique number. A helpful construction used is the strong type graph. We define the strong type graph as the type graph with self loops. We normally do not consider self-loops,  in zero-divisor graphs and type graphs,  but in the strong type graph,  a vertex has a loop at it if it annihilates itself. We denote the strong type graph of $\Gamma(\mathbb{Z}_n)$ as $\Gamma^S(\mathbb{Z}_n)$.\\
Another construction used commonly in this section is $n^*$. Consider some $\Gamma(\mathbb{Z}_n)$. Let $n = p_1^{\alpha_1}p_2^{\alpha_2}\cdots p_m^{\alpha_m}$,  then $n^*=p_1^{\beta_1}p_2^{\beta_2}\cdots p_m^{\beta_m}$ where $\beta_i$ is half of $\alpha_i$ rounded up. This construction is very useful,  as some properties of vertices can be associated with whether or not the vertex is a multiple of $n^*$.

\begin{lemma}\label{2.0} Two arbitrary vertices $u$ and $v$ in $\Gamma(\mathbb{Z}_n)$ that are both in the same type class $T_i$ share an edge iff $T_i$ has a self-loop in the strong type graph.
\end{lemma}
\begin{proof}
Let $T_i$ have a self-loop. Then $i^2 = 0$. Since every $u, v\in T_i$ are multiples of $i$,  $u$ and $v$ will share an edge.\\
Conversely,  let $T_i$ does not have a self-loop. Take arbitrary $u$ and $v$ in $T_i$. According to the definition of type class,  $u$ and $v$ are some multiple of $i$ where $gcf(u, n)=i$ and likewise for $v$. We can write $u=ai$ and $v=bi$ where $gcf(a, n/i)=1$ and $gcf(b, n/i)=1$. Assume $u$ and $v$ share and edge. Then $uv=cn$,  $abi^2=cn$ where $c$ is a natural number. So $\frac{abi^2}{n}=c$. Since $T_i$ does not have a self-loop,  $i^2\neq 0$ which means $n$ contains a factor not contained by $i^2$. Let this factor be called $d$. Let $\frac{g}{d}$ represent the simplified form of the fraction $\frac{i^2}{n}$ where $d$ is guaranteed to not be 1. By substitution,  $\frac{abg}{d} = c$. But this is a contradiction since $a$,  $b$ and $g$ do not share a factor with $n/i$,  so cannot cancel the $d$ out of the denominator. Therefore,  the expression cannot be equal to $c$,  a natural number. $u$ and $v$ do not share and edge.
\end{proof}

\begin{theorem}\label{2.1} $\Gamma(\mathbb{Z}_{p^2})$ is complete where $p$ is prime.
\end{theorem} 
\begin{proof}
Take arbitrary zero divisors of $\mathbb{Z}_n$,  $u$ and $v$. $u$ and $v$ must both share a common factor with $n$,  and the only possible factor is $p$ since $p^2$ is zero. So both $u$ and $v$ have a factor of $p$. Then $u$ and $v$ share an edge. $\Gamma(\mathbb{Z}_{p^2})$ is complete.
\end{proof}

\begin{theorem}\label{2.2} $\Gamma(\mathbb{Z}_{p^x})$ where $p$ is prime and $x \geq 3$ is not complete.
\begin{proof}
Let $x\geq 3$.
\begin{enumerate}
\item[Case 1:] $p=2$: $p$ and $3p$ are distinct non-zero zero-divisors that are not connected.
\item[Case 2:] $p\ne 2$: $p$ and $2p$ are distinct non-zero zero-divisors that are not connected.
\end{enumerate}
\end{proof}
\end{theorem}

\begin{theorem}\label{2.3} $\Gamma(\mathbb{Z}_n)$,  where $n \geq 2$ is complete iff $n=p^2$.
\end{theorem}
\begin{proof}
Let $\Gamma(\mathbb{Z}_n)$ be complete. Assume two or more distinct prime factors of $n$ exist. Label the smallest such factor by $p$. Now choose another distinct prime factor of $n$ as $q$. $p$ is a zero divisor and shares an edge with $n/p$. Since $p$ and $q$ are both prime factors of $n$,  $pq\leq n$. Also,  since $p<q$,  $p^2<pq$. So $p^2<pq\leq n$ which means $p^2$ is non-zero and distinct from $p$. $p^2$ shares an edge with $n/p$ so $p^2$ is a distinct zero-divisor that does not share an edge with $p$,  making $\Gamma(\mathbb{Z}_n)$ not complete. So $n$ must only have one prime factor. Then,  by Theorem 2.2,  $\Gamma(\mathbb{Z}_{p^x})$ is not complete if $x\geq 3$. So $x=2$. So when $\Gamma(\mathbb{Z}_n)$ is complete,  $n=p^2$. 
The converse follows by Theorem \ref{2.1}.
\end{proof}

\begin{theorem}\label{2.4} $\Gamma(\mathbb{Z}_n)$ is k-partite if $\Gamma^S(\mathbb{Z}_n)$ is k-partite.
\end{theorem}
\begin{proof}
Let $\Gamma^S(\mathbb{Z}_n)$ be k-partite. Then $\Gamma^S(\mathbb{Z}_n)$ can be partitioned into $k$ disjoint subsets $S_1, S_2, \cdots, S_k$ such that no vertex in the same set share an edge. Partition $\Gamma(\mathbb{Z}_n)$ into a similar grouping $Q_1, Q_2, \cdots, Q_k$ where $u\in Q_i$ iff $u\in T_u \in S_i$. Consider arbitrary $u$ and $v$,  vertices of $\Gamma(\mathbb{Z}_n)$ that are in the same partitioned set $Q_i$.\\ 
\begin{enumerate}
\item[Case 1:] $u$ and $v$ are in different type classes.\\
Call such classes $T_u$ and $T_v$. Then since $u$ and $v$ are both in $Q_i$,  $T_u$ and $T_v$ are both in $S_i$ which means $T_u$ does not share an edge with $T_v$. So,  by ~\cite{jK55} $u$ and $v$ do not share an edge.\\
\item[Case 2:] $u$ and $v$ are in the same type class.\\
Call this class $T_u$. Then since $\Gamma^S(\mathbb{Z}_n)$ is k-partite,  $T_u$ does not form a loop with itself. Hence,  by Lemma \ref{2.0},  $u$ and $v$ do not share an edge.
\end{enumerate}
\end{proof}

\begin{theorem}\label{2.5} $\Gamma(\mathbb{Z}_n)$ is complete k-partite if $\Gamma^S(\mathbb{Z}_n)$ is complete k-partite.
\begin{proof}
Let $\Gamma^S(\mathbb{Z}_n)$ be complete k-partite. Then by Theorem 2.4,  $\Gamma(\mathbb{Z}_n)$ is k-partite.
Using the partition used in Theorem \ref{2.4},  if we let $\Gamma^S(\mathbb{Z}_n)$ be partitioned into $k$ disjoint subsets $S_1, S_2, \cdots, S_k$,  then $\Gamma(\mathbb{Z}_n)$ can be partitioned into $k$ disjoint subsets $Q_1, Q_2, \cdots, Q_k$,  where arbitrary vertex of $\Gamma(\mathbb{Z}_n)$ is in $Q_i$ if its type class is in $S_i$. Consider arbitrary vertices in $\Gamma(\mathbb{Z}_n)$,  $u$ and $v$,  that are not in the same $Q_i$. Then $u$ and $v$ must be in different  type classes in two different $S_i$'s. Call these classes $T_u$ and $T_v$. Since $\Gamma^S(\mathbb{Z}_n)$ is complete k-partite,  $T_u$ and $T_v$ share an edge. Then $u$ and $v$ share an edge by ~\cite{jK55}.
\end{proof}
\end{theorem}

\begin{remark}The converse of Theorem \ref{2.4} and \ref{2.5} is not always true. If the zero-divisor graph is k-partite,  but has a self-annihilating vertex,  the strong type graph will have a self-loop,  which prevents it from being k-partite.
\end{remark}

\begin{theorem}\label{2.6} If $n$ is square free,  $\Gamma(\mathbb{Z}_{n})$ is k-partite,  where $k$ is the number of distinct prime factors of $n$.
\end{theorem}
\begin{proof}
Consider the strong type graph $\Gamma^S(\mathbb{Z}_{n})$. Let,  $n = p_1p_2\cdots p_k$. Partition the graph into $k$ sets $S_1, S_2, \cdots, S_k$. A vertex $T_a$ in the strong type graph is in $S_i$ if $gcf(a,  p_i) = 1$ and $gcf(a,  p_h) > 1$ for all $h<i$.\\
We now claim that $S_1, S_2, \cdots, S_k$ covers all the vertices of $\Gamma^S(\mathbb{Z}_{n})$.\\
Assume there is a $T_a$ that is not in any $S_i$. Since $T_a$ is a vertex,  $a$ must be a factor of $n$ that is also less than $n$. So $a$ must omit at least one $p_i$. So $gcf(a,  p_i) = 1$. Since $T_a$ is not in any $S_i$,  there must exist some $h<i$ such that $gcf(a,  p_h) = 1$. Choose the smallest index $h$ of such $p_h$. Then $T_a$ must be in $S_h$ which is a contradiction.\\
Our next claim is any two vertices $u$ and $v$ in the same partition do not share an edge.\\
Consider arbitrary $u$ and $v$ in $S_i$. Both $u$ and $v$ do not contain $p_i$ so they do not share an edge. So the strong type graph is k-partite.\\
By Theorem \ref{2.4},  $\Gamma(\mathbb{Z}_{p_1p_2\cdots p_k})$ is k-partite.
\end{proof}

\begin{lemma}\label{2.7} Arbitrary type class $T_a$ in $\Gamma^T(\mathbb{Z}_n)$ contains only one element iff $a=\frac{n}{2}$.
\end{lemma}
\begin{proof}
Let $T_a \in \Gamma(\mathbb{Z}_n)$ have a type class that has only one element. Assume $a\neq \frac{n}{2}$. Since $a$ is a factor of $n$,  $\frac{n}{a}=f$ is also a factor of $n$. Note that $f \geq 3$.\\
Consider the vertex $a(f-1)$. The quantity $(f-1)$ does not share any factors with $f$. Since $af = n$,  $gcf(a(f-1), n) = a$. So $a(f-1)\in T_a$. Also note that $a < a(f-1) < n$. So $a(f-1)$ is a distinct vertex in $T_a$ which is a contradiction.
So $a= \frac{n}{2}$\\
\\
Let $a= \frac{n}{2}$. Then a is the only element in $T_a$ since $2a = n$.
\end{proof}

\begin{corollary}
Analogous to above,  $T_{n/p}$ in $\Gamma^T(\mathbb{Z}_n)$ contains exactly $p - 1$ elements if $p$ is the smallest prime factor of $n$.
\end{corollary}
\begin{lemma}\label{2.8} There is at most one type class with only one element.
\begin{proof}
Assume there are two or more distinct type classes that have only one element. Call two of these classes $T_u$ and $T_v$. By Lemma \ref{2.7},  $u=v=\frac{n}{2}$ which is a contradiction. 
\end{proof}
\end{lemma}

\begin{theorem}\label{2.9} $\Gamma(\mathbb{Z}_n)$ is k-partite if $\Gamma^S(\mathbb{Z}_n)$ is k-partite or $\Gamma^T(\mathbb{Z}_n)$ is k-partite and the only self-connected vertex of $\Gamma(\mathbb{Z}_n)$ is $T_\frac{n}{2}$.
\end{theorem}
\begin{proof}
Let $\Gamma^S(\mathbb{Z}_n)$ be k-partite. By Theorem \ref{2.4},  $\Gamma(\mathbb{Z}_n)$ is k-partite. Let $\Gamma^T(\mathbb{Z}_n)$ be k-partite and let $\Gamma^S(\mathbb{Z}_n)$ have only one self-connected vertex,  $T_\frac{n}{2}$. Consider arbitrary distinct $u$ and $v$,  zero divisors of $\Gamma(\mathbb{Z}_n)$,  that are in the same partition.\\
\begin{enumerate}
\item[Case 1:] $u$ and $v$ are in the same type class.\\
By Lemma \ref{1.10},  $T_\frac{n}{2}$ has only one element,  so if $u$ and $v$ are distinct,  they cannot be in $T_\frac{n}{2}$. Then the type class they are in are not self-connected so $u$ and $v$ do not share an edge.\\
\item[Case 2:] $u$ and $v$ are in different type classes.\\
Since $u$ and $v$ are in the same partition,  their type classes are in the same partition and do not share an edge. Thus,  $u$ and $v$ do not share an edge.\\
\end{enumerate}
\end{proof}

\begin{lemma}\label{2.10} A vertex in $\Gamma(\mathbb{Z}_n)$ annihilates itself iff it is a multiple of $n^*$.
\end{lemma}

\begin{lemma}\label{2.11} Consider two arbitrary vertices in $\Gamma(\mathbb{Z}_n)$ $u$ and $v$ such that $u$ is a factor of $v$. The largest clique containing $v$,  $M_v$ has a magnitude greater than or equal to the $M_u$,  the largest clique containing $u$.
\end{lemma}
\begin{proof}
Take arbitrary vertices $u$ and $v$ in $\Gamma(\mathbb{Z}_n)$. Let $u$ be a factor of $v$. Assume the opposite,  that $M_u$ has a larger magnitude than that of $M_v$. Every element $e$ in $M_u\setminus u$ has the property $eu=0$. Then $\forall e\in M_u$,  $ev=0$. So a clique $C$ exists with $v$ and each $e$ in $M_u\setminus u$. $C$ has a magnitude equal to the magnitude of $M_u$ which is a contradiction since $M_v$ is the largest clique containing $v$.
\end{proof}

\begin{theorem}\label{2.12} $cl(\Gamma(\mathbb{Z}_n)) \geq \frac{n}{n^*} + k - 1$ where $k$ is the number of odd-power primes in the prime factorization of $n$. 
\end{theorem}
\begin{proof}
We claim that any two multiples of $n^*$ share an edge.\\ Take two arbitrary multiples of $n^*$,  $an^*$ and $bn^*$. Since $(n^*)^2 \geq n$ these two vertices will share an edge. So the multiples of $n^*$ form a clique. Call it $C$. An arbitrary vertex of $C$ will be of the form $an^*$ for $1<a<\frac{n}{n^*}$. The amount of elements in this clique is $\frac{n}{n^*} - 1$,  so the clique number of the graph is at least $\frac{n}{n^*} - 1$. Now consider all vertices of the form $n^*/q$ where $q$ is an arbitrary odd-power prime in the prime factorization of $n$. Because $n^*$ has a factor of $q$ with power of half rounded up,  and $n^*/q$ has a power of half rounded down,  arbitrary $n^*/q$ shares an edge with each $an^*$ in $C$. Also,  each $n^*/q_1$ shares an edge to each other $n^*/q_2$. This is because the power of $q_1$ in $n^*/q_1$ is half rounded down and in $n^*/q_2$ it is half rounded up,  and likewise for $q_2$. Since $k$ is the number of distinct odd powered primes in the prime factorization of $n$,  $cl(\Gamma(\mathbb{Z}_n)) \geq \frac{n}{n^*} + k - 1$.
\end{proof}

\begin{theorem}\label{2.13} $cl(\Gamma(\mathbb{Z}_n)) \leq \frac{n}{n^*} + k - 1$ where $k$ is the number of odd-power primes in the prime factorization of $n$. 
\end{theorem}
\begin{proof}
Consider arbitrary clique $C$. Partition $C$ into sets $L$ and $N$ where $L$ is the set of vertices of $C$ that are not multiples of $n^*$ and $N$ is the set of vertices of $C$ that are multiples of $n^*$. Consider arbitrary vertex $l_1$ in $L$. Since $l_1$ is not a multiple of $n^*$,  there must be some prime factor $p_1$ of $n$ whose power in $l_1$ is less than half of its power in $n$ (since if every prime factor was greater than or equal to half,  $l_1$ would be a multiple of $n^*$). Every other $l_i$ in $L$ must have its $p_1$ factor with a power greater than or equal to half its power in $n$ for it to share an edge with $l_1$. Consider another vertex $l_2$ in $L$. $l_2$ must also have a prime factor whose power is less than half its power in $n$,  but it cannot be $p_1$. Call it $p_2$. So each $l_i$ in $L$ must have a distinct prime factor $p_i$ that has a power less that or equal to half its power in $n$. Let $m$ be the number of distinct prime factors of $n$. Then there can be a maximum of $m$ many $l_i$ in $L$. $N$ has a maximum size of $\frac{n}{n^*}-1$,  so the clique number is at most $\frac{n}{n^*}+m-1$.\\
Consider some $e_1$,  a vertex in $L$ whose corresponding $p_1$,  has an even power in $n$. $e_1$ does not share an edge with $n^*$. This means the clique number is one less if $n$ has an even-powered prime. Consider another $e_2$ that has an even $p_2$ whose power is less than half. Then $e_2$ does not share an edge with $p_1 n^*$. In general,  a vertex $e_i$ whose corresponding $p_i$ has an even power does not share an edge with distinct vertices $p_1 p_2 \cdots p_{i-1} n^*$. So the size of $C$ is reduced by the number of even powered-primes of $n$. This value can be represented by $m-k$ where $k$ is the number of odd-powered primes of $n$. Hence,  since $C$ is arbitrary,  $cl(\Gamma(\mathbb{Z}_n)) \leq \frac{n}{n^*} + m - (m - k) - 1$. $cl(\Gamma(\mathbb{Z}_n)) \leq \frac{n}{n^*} + k - 1$.
\end{proof}

\begin{theorem}\label{2.14} $cl(\Gamma(\mathbb{Z}_n)) = \frac{n}{n^*} + k - 1$.
\end{theorem}

\begin{proof}
The proof follows by Theorem \ref{2.12} and Theorem \ref{2.13}.
\end{proof}

\begin{theorem}\label{2.15} There are no non-empty,  non-complete,  regular $\Gamma(\mathbb{Z}_n)$.
\end{theorem}
\begin{proof}
Consider all $\Gamma(\mathbb{Z}_n)$ that are non-empty and not complete. Assume $\exists$ some regular graph among these graphs.\\
\begin{enumerate}
\item[Case 1:] $n=p^x$ where $p$ is prime\\
If $x=1$,  the graph is empty,  and if $x=2$,  the graph is complete,  so $x\geq 3$.
Then $p$ is a vertex that shares an edge with $p-1$ many other vertices,  and $p^2$ is a vertex that shares an edge with $p^2-1$ many other vertices. Since the graph is regular,  $p-1=p^2-1$,  thus $p=p^2$,  which means $p=1$,  a contradiction.\\
\item[Case 2:] $n=p_1^{\alpha_1}p_2^{\alpha_2}\cdots p_m^{\alpha_m}$,  $m \geq 2$ and $p_i$ are all prime\\
Vertex $p_1$ shares an edge with $p_1-1$ many other vertices,  and the vertex $p_2$ shares an edge with $p_2-1$ many other vertices. Since the graph is regular,  $p_1-1=p_2-1$,  thus $p_1=p_2$ which is a contradiction since $p_1$ and $p_2$ are distinct.\\
So the only non-empty regular graphs are complete.
\end{enumerate}
\end{proof}

\begin{theorem}\label{2.16} $\Gamma(\mathbb{Z}_n)$ is chordal iff $n=p^x,  2p$ or $2p^2$,  where $p$ is prime and $x$ is a positive integer.
\end{theorem}
\begin{proof}
Let $n=p^x$. Assume that $\Gamma(\mathbb{Z}_{p^x})$ is not chordal. Then $\exists$ a cycle $C$ of length $>$ 3,  that has no chord. Let $y$ be a vertex of $C$ that is not a multiple of $n^*$. Then,  since the power of $p$ in $y$ has a power strictly less than $\frac{x}{2}$,  each neighbor must be a multiple of $n^*$. Then the two neighbors of $y$ in $C$ share an edge which is a chord. So all vertices in $C$ must be a multiple of $n^*$ which also causes a chord. So $\Gamma(\mathbb{Z}_{p^x})$ is chordal.\\

Let $n=2p$. $\Gamma(\mathbb{Z}_{2p})$ is a star because it is a line segment only. Then,  $\Gamma(\mathbb{Z}_{2p})$ is chordal.\\

Let $n=2p^2$. Assume $\Gamma(\mathbb{Z}_{2p^2})$ is non-chordal. Then $\exists$ a cycle $C$ of length $>$ 3 that has no chord.\\
Let $a$ be a vertex of $C$ in the type class $T_p$. Each neighbor of $a$ must be a multiple of $2p$,  and therefore,  is in the type class $T_{2p}$. Each multiple of $2p$ shares an edge,  so there exists a chord in $C$. So there can be no vertices in the type class $T_p$ in $C$.\\
Let $b$ be a vertex of $C$ that is in the type class $T_2$. Every neighbor of $b$ must be in the type class $T_{p^2}$. But there is only one element in $T_{p^2}$ so $b$ cannot have two distinct neighbors. So $b$ is not a vertex of $C$.\\
So each vertex of $C$ must be in either $T_{p^2}$ or $T_{2p}$. Then since there is only one element of $T_{p^2}$,  and the magnitude of $C$ is at least 4,  there are at least 3 elements of $T_{2p}$ in $C$. Those 3 elements form a triangle since each multiple of $2p$ annihilates each other multiple of $2p$. But $C$ can't have a triangle since it is chord-less. This is a contradiction.\\
\\
Let $n$ not be $p^x$,  $2p$ or $2p^2$.\\
\begin{enumerate}
\item[Case 1:] $n=2^xp^y$ where $y\geq 3$,  $x \geq 1$ and $p$ is an odd prime.\\
Then $2^xp-p^y-2^{x+1}p-p^{y-1}$ is a chord-less cycle.\\
\item[Case 2:] $n=2^xp^y$ where $x\geq 2$,  $y \geq 1$ and $p$ is an odd prime.\\
Then $2p^y-2^x-p^y-2^{x+1}$ is a chord-less cycle.\\
\item[Case 3:] $n=p^xq^y$ where $p, q\geq 3$ where $p \neq q$ are primes and $x, y$ are non-zero.\\
Then $p^x-q^y-2p^x-2q^y$ is a chord-less cycle.\\
\item[Case 4:] $n=p_1^{\alpha_1}p_2^{\alpha_2}\cdots p_k^{\alpha_k}$ where $k\geq 3$ and $\alpha_i$ is non-zero.\\ Since $k\geq 3,  n$ has an odd prime factor $p_1$.
 Then $p_1^{\alpha_1}-n/p_1^{\alpha_1}-2p_1^{\alpha_1}-2n/p_1^{\alpha_1}$ is a chord-less cycle.\\
So $\Gamma(\mathbb{Z}_n)$ is non-chordal if $n$ is not $p^x$,  $2p$ or $2p^2$.
\end{enumerate}
\end{proof}

\begin{lemma}\label{2.17} If $n^* \neq n$,  $\Gamma(\mathbb{Z}_n)$ has a simplicial vertex.

\end{lemma}
\begin{proof}
Let $n^* \neq n$. Then $n/n^*$ is a vertex since $n/n^*$ shares an edge with $n^*$ which is not a multiple of $n$. Since every neighbor of $n/n^*$ is a multiple of $n^*$ and every multiple of $n^*$ shares an edge,  $n/n^*$ is a simplicial vertex. So $\Gamma(\mathbb{Z}_n)$ has a simplicial vertex.
\end{proof}

Another construction $n_*$ can be useful. It is similar to $n^*$,  but for the odd powered primes,  round down instead of up. Consider $\Gamma(\mathbb{Z}_n)$ where $n=p_1^{\alpha_1}p_2^{\alpha_2}\cdots p_k^{\alpha_k}$. Define $n_*$ as $n_* = p_1^{\beta_1}p_2^{\beta_2}\cdots p_k^{\beta_k}$ and $\beta_i = \alpha_i/2$ if $\alpha_i$ is even and $\beta_i = (\alpha_i-1)/2$ if $\alpha_i$ is odd.\\
\\
Note that $n_*n^* = 0$ and if $n$ is square-free,  $n_* = 1$.

\begin{lemma}\label{2.18} Arbitrary vertex $v$ in $\Gamma(\mathbb{Z}_n)$ is a simplicial vertex iff $v\in T_2$ or $v\in T_g$ where $g$ is a factor of $n_*$.
\end{lemma}
\begin{proof}
Take arbitrary $v$ in $\Gamma(\mathbb{Z}_n)$. Let $v\in T_2$. Then $v$ only shares an edge with vertices in $T_{n/2}$. By Lemma \ref{2.7},  $T_{n/2}$ has only one element,  which makes a clique. So $v$ is simplicial.\\
Let $v\in T_g$ where $g$ is a factor of $n_*$. So $n_* = ag$ where $a$ is a positive integer. Consider some vertex $h$ in $T_j$ that shares an edge with $v$. Then $j*g = bn$ for positive integer $b$. $\frac{jn_*}{a} = bn_*n^*$. $\frac{j}{a} = bn^*$. Then $j=abn^*$. So $j$ is a multiple of $n^*$ and therefore,  $h$ is a multiple of $n^*$. Since every multiple of $n^*$ shares an edge with every other such multiple,  $v$ is a simplicial vertex.\\
Conversely,  let $v$ be neither in $T_2$ nor in any $T_g$ where $g$ is a factor of $n_*$. Then,  since $v$ is not in any $T_g$,  $v$ has some prime with a power greater than half of that in $n$. Call that prime $p_x$ and its power in $v$,  $\alpha_x$. Let the type class of $v$ be called $T_w$. Consider the type class $T_{n/w}$. Each vertex in $T_{n/w}$ shares an edge with $v$. Since $v\notin T_2$,  $T_{n/w} \neq T_{n/2}$. So by Lemma \ref{2.7},  $T_{n/w}$ has more than one element. Since $n/w$ has a power of $p_x$ less than that of half in $n$,  none of the vertices in $T_{n/w}$ share an edge with each other. So the neighbors of $v$ do not form a clique. Hence,  $v$ is not simplicial. 
\end{proof}

\begin{theorem}\label{2.19} $\Gamma(\mathbb{Z}_n)$ has a simplicial vertex iff the prime factorization of $n$ is not square free or $n$ is even.
\end{theorem}
\begin{proof}
Let $n$ not be square free. Then,  $n^* \neq n$. So by Lemma \ref{2.17},  $\Gamma(\mathbb{Z}_n)$ has a simplicial vertex.\\
\\
Let $n$ be even. Then,  2 is a zero divisor. Every neighbor of 2 must be a multiple of $n/2$ which there is only one of,  so 2 is a simplicial vertex.\\
\\
Let $n$ be square free and odd. 2 is therefore not a factor of $n$. Then consider arbitrary vertex $x$. $x$ shares an edge with both $n/x$ and $2n/x$. $2n/x$ is non-zero since $x$ is necessarily odd,  and $n/x$ and $2n/x$ do not share an edge since $n$ is odd. For if $\frac{n}{x}\frac{2n}{x} = ny$,  $2n = yx^2$ and $n = \frac{yx^2}{2}$ which is a contradiction. So there are no simplicial vertices of $\Gamma(\mathbb{Z}_n)$.
\end{proof}

Note: It follows by ~\cite{jK55},  (observation 3.2),  if in $\Gamma(\mathbb{Z}_n)$ a vertex $u$ is simplicial then $T_u$ is simplicial in $\Gamma^{T}(\mathbb{Z}_n)$. But,  not conversely. For example,  in $\Gamma^{T}(\mathbb{Z}_{12}),  T_3$ is simplicial,  where as $3$ is not so in $\Gamma(\mathbb{Z}_{12})$.

\begin{lemma}\label{2.20} If $\Gamma(\mathbb{Z}_n)$ has three or more prime factors of $n$,  $\Gamma(\mathbb{Z}_n)$ is not $\gamma-\beta$ perfect.
\end{lemma}
\begin{proof}
Let $n=p_1^{\alpha_1}p_2^{\alpha_2}\cdots p_k^{\alpha_k}$ where $k\geq 3$. By ~\cite{bH77},  the domination number is $k$. Construct some vertex map $V$ whose size is $k$.\\
We claim that $V$ must contain the vertex $n/p_x$ for every $p_x$ prime factor of $n$.\\
Consider the vertex $n/p_x$ for some $p_x$ prime factor of $n$. Let $n/p_x$ not be in $V$. Construct set $C = \{ p_xp_i | 1\leq i\leq k \}$. $n/p_x$ shares an edge with every vertex in $C$. Since $n/p_x \notin V$,  every element of $C$ is in $V$. $C$ has $k$ many vertices,  so $V$ has at least $k$ many vertices. Consider vertex $p_x$. $p_x$ shares an edge with $n/p_x$ which is not covered by $V$,  so $V$ has at least $k+1$ vertices. That is a contradiction since the size of $V$ is $k$. So each $n/p_x$ is in $V$.\\
\\
Consider the type classes $T_{n/p_1}$,  $T_{n/p_2}$ and $T_{n/p_3}$. By Lemma \ref{2.8},  there can be at most one type class with only one element. At least two of these type classes have more than one element. Without loss of generality,  let them be $T_{n/p_1}$ and $T_{n/p_2}$. Since $n/p_1$ and $n/p_2$ are both in $V$,  choose different vertices in the type classes $u$ and $v$. $u$ and $v$ share an edge since they are multiples of $n/p_1$ and $n/p_2$ respectively,  so they share an edge,  but are not in $V$. Then $V$ must contain at least one other element making the size of $V$ at least $k+1$. This is a contradiction. We cannot construct a vertex map size $k$. So $\Gamma(\mathbb{Z}_n)$ is not $\gamma-\beta$ perfect.
\end{proof}

\begin{theorem}\label{2.21} The only $\gamma-\beta$ perfect $\Gamma(\mathbb{Z}_n)$ are $n=2^3, 3^2, p, 2p$ and $3p$.
\end{theorem}
\begin{proof}
Let $n=2^3$. The domination number clearly equals the smallest vertex map.\\

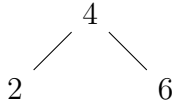
\begin{tikzpicture}

\node (2) at (1, 1) {2};
\node (4) at (2, 2) {4};
\node (6) at (3, 1) {6};

\foreach \from/\to in {2/4, 4/6}
 \draw (\from) -- (\to);

\end{tikzpicture}\\

Let $n=3^2$. The domination number clearly equals the smallest vertex map.\\

\begin{tikzpicture}

\node (3) at (1, 1) {3};
\node (6) at (2, 1) {6};

\foreach \from/\to in {3/6}
 \draw (\from) -- (\to);

\end{tikzpicture}\\

Let $n=2p$. Then the graph is a star,  so the domination number and the smallest vertex map are both 1.\\
Let $n=3p$. Then $V = \{p,  2p\}$ is both a minimal dominating set and a minimal vertex map.\\
Let $n=p$. Then both the domination number and the smallest vertex map is 0 since the graph is empty.\\
\\
Now,  we will show that all other $\Gamma(\mathbb{Z}_n)$ are not $\gamma-\beta$ perfect.\\\\

Let $n=2^2$. The empty set is a vertex map since there are no edges in this map,  so smallest vertex map and the domination number are not the same.\\
Let $n=2^x$,  $x\geq 4$. Then $2^{x-1}-2^{x-2}-3\cdot2^{x-2}$ is a triangle. Triangles prevent vertex maps of size 1,  and by ~\cite{bH77} the domination number is 1,  so the values do not match.\\
Let $n=3^x$,  $x\geq 3$. Then $3^{x-1}-2\cdot3^{x-1}-3^{x-2}$ is a triangle that prevents vertex maps of size 1.\\
Let $n=p^x$,  $p\geq 5$,  $x\geq 2$. Then $p^{x-1}-2\cdot p^{x-1}-3\cdot p^{x-1}$ is a triangle.\\
Let $n=pq$,  $q>p\geq 5$. The domination number is 2 by ~\cite{bH77}. $p-q-2p-2q-3p-2q$ is a hole size 6. There cannot be a vertex map that covers a hole of that size,  so the smallest vertex map is not 2.\\
Let $n=p^xq$,  $x\geq 2$. The domination number is 2.\\
\begin{enumerate}
\item[Case 1:] $p=2$.\\
Then $p^{x-1}q-p^x-q-p^{x+1}-pq$ is a non-induced sub-graph that cannot be covered by a vertex map size 2.\\
\item[Case 2:] $p\neq 2$.\\
Then $p^x-p^{x-1}q-p-2p^{x-1}q-2p$ is a non-induced sub-graph that cannot be covered by a vertex map size 2.
The smallest vertex map is larger than 2 making the graph not $\gamma-\beta$ perfect.\\
Let $n=p^xq^y$,  $x, y \geq 2$. The domination number is 2 by ~\cite{bH77}. Assume there is a vertex map $V$ size 2. Consider the edges $p-p^{x-}q^y$ and $q-p^xq^{y-1}$. $V$ must contain at least vertex one of each edge. By Lemma \ref{2.8} only one type class can have only one vertex. Consider the type classes $T_{p^xq^{y-1}}$ and $T_{p^{x-1}q^y}$. At least one of them must contain more than one vertex. Without loss of generality let that be $T_{p^{x-1}q^y}$. Then there exists some $u\in T_{p^{x-1}q^y}$ that is not in $V$. The edge $p-u$ is not covered by $V$,  so the size of $V$ is at least one more than 2 which is a contradiction.\\
Let $n=p_1^{\alpha_1}p_2^{\alpha_2}\cdots p_k^{\alpha_k}$,  $k\geq 3$. Then by Lemma \ref{2.20},  the graph is not $\gamma-\beta$ perfect.\\
So the only $\gamma-\beta$ graphs $\Gamma(\mathbb{Z}_n)$ are $2^3,  3^2,  p,  2p$ and $3p$.
\end{enumerate}
\end{proof}

\section{Some properties of $\Gamma(\mathbb{Z}_{n_1}\times\cdots\times\mathbb{Z}_{n_k})$}

In this section,  we discuss some facts about $\Gamma(\mathbb{Z}_{n_1}\times\cdots\times\mathbb{Z}_{n_k})$. It is often possible to relate some properties of the individual $\Gamma(\mathbb{Z}_{n_i})$ to the graph of the product. One example is that the domination number of $\Gamma(\mathbb{Z}_{n_1}\times\cdots\times\mathbb{Z}_{n_k})$ has an upper and lower bound corresponding to the domination number of each $\Gamma(\mathbb{Z}_{n_i})$.

\begin{theorem}\label{3.0} Consider two arbitrary commutative rings with unity,  $R$ and $S$. $\Gamma(R\times S)$ is complete iff $|R|=|S|=2$.
\end{theorem}
\begin{proof}
Consider some $R$ and $S$ such that $|R|=|S|=2$. Since both $R$ and $S$ have $1$,  the only elements of $R$ and $S$ are $0$ and $1$,  where by $1$ we denote the unity of the respective ring . Then the zero divisor graph is $(0, 1) - (1, 0)$ which is complete.\\
Conversely,  let $R$ or $S$ have more than 2 elements. Without loss of generality,  let $R$ have more than 2 elements. Then $R$ has some element $a$ that is neither $1$ nor $0$. The graph $\Gamma(R\times S)$ has vertices $(1,  0)$ and $(a,  0)$. These vertices do not share an edge because $1\cdot a = a$ which is not zero. So $\Gamma(R \times S)$ is not complete.\\
\end{proof}

\begin{theorem}\label{3.1} $\Gamma(R_1\times \cdots\times R_k)$ where $k \geq 2$ and each $R_i$ is a commutative ring with $1$. This graph is complete iff $k=2$ and $|R_i| = 2$ for all $i$.
\end{theorem}
\begin{proof}
Consider some $\Gamma(R_1\times\cdots\times R_k)$ where $k=2$ and all $|R_i| = 2$. Then by Theorem 3.0,  $\Gamma(R_1\times\cdots\times R_k)$ is complete.\\
Consider some $\Gamma(R_1\times \cdots\times R_k)$ that does not meet this criteria. If $k \geq 3$,  then $(1,  0,  1)$ and $(1,  1,  0)$ are two vertices that do not share an edge. If any $|R_i| \geq 2$,  then $R_i$ has an element $a$ that is not 0 or 1. Then $(\cdots, a, \cdots)$ does not share an edge with $(\cdots, 1, \cdots)$,  where $a$ and $1$ are placed in the $i-th$ entry of the respective elements. So $\Gamma(R_1\times \cdots\times R_k)$ is not complete.
\end{proof}

\begin{theorem}\label{3.2} $\Gamma(\mathbb{Z}_n\times\mathbb{Z}_m)$ where $n, m \geq 2$ is complete-bipartite iff $n$ and $m$ are prime.
\begin{proof}
Let $m$ and $n$ be prime. Then partition $\Gamma(\mathbb{Z}_n\times\mathbb{Z}_m)$ into sets $S_n$ and $S_m$ such that
$S_n=\{(x, 0)|0<x<n\}$ and $S_m=\{(0, y)|0<y<m\}$.\\
We claim that $S_n\cup S_m = \Gamma(\mathbb{Z}_n\times\mathbb{Z}_m)$.\\
Assume,  $\exists$ a zero divisor $a=(a_1, a_2)$ that is not in $S_n\cup S_m$. Both $a_1$ and $a_2$ are non-zero as $m$ and $n$ are prime. Since $a$ is a zero-divisor,  there must be some $b = (b_1, b_2)$ that shares an edge with $a$. So $a_1 b_1 = 0$. Since $\mathbb{Z}_n$ has no non-zero divisors,  and $a_1$ is not zero,  $b_1 = 0$. In the same way we find that $b_2$ is zero. This means $a$ is not a zero-divisor because it only shares an edge with 0. So $S_n\cup S_m = \Gamma(\mathbb{Z}_n\times\mathbb{Z}_m)$.\\
Take arbitrary $u, v\in S_n$. Then $u=(u_1, 0)$ and $v=(v_1, 0)$. Since $u_1 v_1 \neq 0$,  $uv\neq (0, 0)$ which means $u$ and $v$ do not share an edge. In the same way $u$ and $v$ do not share an edge if they are both in $S_m$. So $u$ and $v$ do not share an edge if they are in the same partition which is the definition of bipartite.\\
Thus,  it follows from the construction of $S_m$ and $S_m$,  that $\Gamma(\mathbb{Z}_n\times\mathbb{Z}_m)$ is complete bipartite.\\
Conversely,  let $\Gamma(\mathbb{Z}_n\times\mathbb{Z}_m)$ be complete bipartite. Assume one or both n and m are not prime. Let the non-prime be $n$. Then,  there is a non-zero zero divisor of $\mathbb{Z}_n$. Call it $k$. Since $\Gamma(\mathbb{Z}_n\times\mathbb{Z}_m)$ is complete-bipartite,  the vertices of $\Gamma(\mathbb{Z}_n\times\mathbb{Z}_m)$ can be partitioned into 2 disjoint subsets such that no edges exist between two vertices in the same partition,  and every pair of vertices in different partitions share an edge. $(1, 0)$ is a zero divisor since it shares an edge with $(0, 1)$. $(k, 0)$ is also a zero divisor since it also shares an edge with $(0, 1)$. Since $(k, 0)$ does not share an edge with $(1, 0)$,  they must be in the same partition. Call it $S_1$ and let the other partition be $S_2$. Since $k$ is a zero-divisor of $\mathbb{Z}_n$,  $\exists k'$ not necessarily distinct such that $k\cdot k' = 0$. Then $(k',  1)$ shares an edge with $(k,  0)$ which means $(k',  1)\in S_2$. Since $\Gamma(\mathbb{Z}_n\times\mathbb{Z}_m)$ is complete-bipartite,  $(1, 0)$ must share an edge with $(k', 1)$ since they are in opposite partitions,  but their product is not $0$,  which is a contradiction. So both $n$ and $m$ must be prime.
\end{proof}
\end{theorem}

\begin{corollary} From this theorem it follows that $\Gamma(\mathbb{Z}_n\times\mathbb{Z}_m)$ has a complete bipartite sub-graph. 
\end{corollary}
This is formed by $S_n \cup S_m$. If one of them is not a prime,  we can delete all vertices that has at least one  an entry dividing either $n$ or $m$ respectively,  to get a complete bipartite subgraph.

\begin{theorem}\label{3.3} $\Gamma(\mathbb{Z}_{n_1}\times\cdots\times\mathbb{Z}_{n_k})$ where $\forall n_i \geq 2$ and $k\geq 2$ is bipartite iff $k=2$ and both $n_i$ are prime,  or one $n_x$ is prime and the other is $4$.
\end{theorem}
\begin{proof}
Let $k=2$ and both $n_1$ and $n_2$ be prime. By Theorem \ref{3.2},  $\Gamma(\mathbb{Z}_{n_1}\times \mathbb{Z}_{n_2})$ is bipartite.\\
Let $k=2$ and let one of $n_i$ be 4 and the other be prime. Without loss of generality,  let $n_1=4$. Then $n_2$ is prime. Partition the vertices into set $A$ and $B$ where $A$ is the set of all vertices of the form $(a,  0)$ where $a\in\mathbb{Z}_{n_1}/0$ and $B$ is everything else. Consider arbitrary,  distinct elements of $A$,  $(a_1,  0)$ and $(a_2,  0)$. They do not share an edge,  since there are no two distinct $a_1$ and $a_2$ that share an edge in $\Gamma(\mathbb{Z}_{n_1})$. Consider all vertices in $B$. Assume $\exists u, v \in B$ such that $u$ shares an edge with $v$. Then,  $u = (u_1, u_2)$ and $v = (v_1, v_2)$. Note that $u_2v_2 \neq 0$. $u_2 v_2 = 0$ which means $u_2$ and $v_2$ are zero divisors in $\Gamma(\mathbb{Z}_{n_2})$. This is impossible since there are no zero divisors in $\Gamma(\mathbb{Z}_{n_2})$. So $\Gamma(\mathbb{Z}_{n_1}\times\cdots\times\mathbb{Z}_{n_k})$ is bipartite.\\
Conversely,  let $\Gamma(\mathbb{Z}_{n_1}\times\cdots\times\mathbb{Z}_{n_k})$ be bipartite.\\
We first claim that $k=2$.\\
Assume $k\geq 3$. Then,  $(1,  0,  0, \cdots,  0) - (0,  1,  0, \cdots,  0) - (0,  0,  1, \cdots,  0)$ is a triangle which cannot exist in a bipartite graph. So $k < 3$. By our definition,  $k\geq 2$,  so $k=2$\\
We now claim no $\Gamma(\mathbb{Z}_{n_i})$ can have two or more distinct zero divisors.\\
Assume otherwise. Call two such divisors $u$ and $v$ that share an edge in $\Gamma(\mathbb{Z}_{n_i})$. Without loss of generality,  let $u$ and $v$ be in the first slot (so $i=1$). Then $(u,  0) - (v,  0) - (0,  1)$ is a triangle which cannot exist in a bipartite graph. The only $\Gamma(\mathbb{Z}_{n_i})$ that has one element is $\Gamma(\mathbb{Z}_4)$. So all $n_i$ must be either 4 or prime.\\
Our final claim is it is not possible for both $n_i$ to be 4.\\
Assume otherwise. Then $(2,  0) - (2,  2) - (0,  2)$ is a triangle which cannot exist in a bipartite graph.
So,  because $\Gamma(\mathbb{Z}_{n_1}\times\mathbb{Z}_{n_2}\times\cdots \times\mathbb{Z}_{n_k})$ is bipartite,  $k=2$ and either both $n_i$ are prime,  or one is 4 and the other is prime.
\end{proof}

\begin{theorem}\label{3.4} $\Gamma(R_1\times\cdots\times R_k)$ where each $R_i$ is a commutative ring with 1 is not perfect if some $\Gamma(R_i)$ is not perfect.
\end{theorem}
\begin{proof}
Let some $\Gamma(R_i)$ be non-perfect. Then by the Strong Perfect Graph theorem,  there exists an odd hole or anti-hole $H$ of length 5 or greater. Let $H$ have a length $l$. Then we write it as,  $v_1-v_2-\cdots-v_{l-1}-v_l-v_1$. Then a hole exists in $\Gamma(R_1\times\cdots\times R_k)$. Fill in the $i$th position with the vertices of $H$,  and fill the rest in with zeros. The hole is $(0, \cdots, 0, v_1, 0, \cdots, 0)-(0, \cdots, 0, v_2, 0, \cdots, 0)-\cdots-(0, \cdots, 0, v_{l-1}, 0, \cdots, 0)-(0, \cdots, 0, v_l, 0, \cdots, 0)-(0, \cdots, 0, v_1, 0, \cdots, 0)$.\\
The same proof can be used for an anti-hole. So if any $\Gamma(R_i)$ are non-perfect,  $\Gamma(R_1\times\cdots\times R_k)$ will also be non-perfect.
\end{proof}

\begin{note} 
The converse of Theorem 3.4 is not true. In the graph $\Gamma(\mathbb{Z}_2\times\mathbb{Z}_2\times\mathbb{Z}_2\times\mathbb{Z}_2\times\mathbb{Z}_2)$,  every $\Gamma(\mathbb{Z}_2)$ is perfect,  but we find the hole $(1,  1,  0,  0,  0)-(0,  0,  1,  1,  0)-(1,  0,  0,  0,  1)-(0,  1,  1,  0,  0)-(0,  0,  0,  1,  1)$.
\end{note}

\begin{theorem}\label{3.5} $\Gamma(R_1\times\cdots\times R_x)$ where each $R_i$ is a commutative ring with 1 is not regular if any $\Gamma(R_i)$ is not empty.
\end{theorem}
\begin{proof}
Take $\Gamma(R_1\times\cdots\times R_x)$. Let some $\Gamma(R_i)$ be non-empty. Consider the vertex $g=(0, \cdots, 0, 1, 0, \cdots, 0)$ that has a 1 at the $i^{th}$ index and 0 filled in all other indices. All neighbors of $g$ must be of the form $(a_1, a_2, \cdots, a_{i-1}, 0, a_{i+1}\cdots, a_{x-1}, a_x)$,  with a zero at the $i^{th}$ index and any value in the other indices,  not all zero. Let there are  $f$ many such vertices.
Since $\Gamma(R_i)$ is non-empty,  $\exists k\in \Gamma(R_i)$. Since $k$ is a zero divisor,  there must be some $k'\in \Gamma(R_i)$,  not necessarily distinct,  such that $k\cdot k'=0$. Consider the vertex $h=(0, \cdots, 0, k, 0, \cdots, 0)$ with $k$ in the $i^{th}$ index and the rest filled in with 0. This vertex shares an edge with all vertices that share an edge with $g$. So $h$ shares an edge with at least $f$ vertices. But it also shares an edge with $(1, \cdots, 1, k', 1\cdots, 1)$ which means $h$ shares an edge with at least $f+1$ vertices. This means $g$ and $h$ have a different number of neighbors,  so $\Gamma(R_1\times\cdots\times R_x)$ is not regular.
\end{proof}

\begin{theorem}\label{3.6} For arbitrary rings $R$ and $S$,  $cl(\Gamma(R\times S)) \geq cl(\Gamma(R)) + cl(\Gamma(S)) + |R'||S'|$ where $R'$ and $S'$ are any set of self-annihilating vertices in a maximal clique of $\Gamma(R)$ and $\Gamma(S)$.
\end{theorem}
\begin{proof}
Let $C$ be a maximal clique in $\Gamma(R)$ and $D$ be a maximal clique in $\Gamma(S)$. Construct an induced sub graph $X = \{(c, 0) or (0,  d) | c\in C,  d\in D\}$. Take two arbitrary,  distinct vertices in $X$,  call them $u$ and $v$.\\
\begin{enumerate}
\item[Case 1:] $u=(c_1,  0),  v=(c_2,  0)$\\
Since $c_1$ shares an edge with $c_2$,  $u$ and $v$ share an edge.\\
\item[Case 2:] $u=(0,  d_1),  v=(0,  d_2)$\\
since $d_1$ shares an edge with $d_2$,  $u$ and $v$ share an edge.\\
\item[Case 3:] $u$ and $v$ are not of the same form.\\
Then,  without loss of generality,  let $u = (c,  0)$ and $v = (0,  d)$. $u$ shares an edge with $v$.\\
So $X$ is a clique in $\Gamma(R\times S)$ with size $cl(\Gamma(R)) + cl(\Gamma(S))$.\\
Now consider $R'$,  the set of all self-annihilating vertices in $C$. Each vertex in $R'$ shares an edge with each other vertex in $R'$ because it is an induced sub-graph of a clique. It also shares an edge with every vertex in $C$. Likewise,  every vertex in $S'$,  the set of all self-annihilating vertices in $D$,  shares an edge with every other vertex in $S'$ and every vertex in $D$. Define the induced sub-graph $Y = \{ (r, s) | r\in R',  s\in S' \}$. Every vertex $(r, s) \in Y$ shares an edge with every other vertex in $Y$ and every vertex in $X$,  so $X\cup Y$ forms a clique size $cl(\Gamma(R)) + cl(\Gamma(S)) + |R'||S'|$.
\end{enumerate}
\end{proof}

\begin{corollary}
	Consider $n$ many arbitrary rings $R_1, R_2, \cdots R_n$. Then, 
	$cl(\Gamma(R_1\times R_2\cdots R_n)) \geq \sum_{i=1}^{n}cl(\Gamma(R_i))+\sum_{i\neq j,  i, j \in \{1, 2, \cdots n\}}|R_{i}'||R_{j}'|+\sum_{i\neq j\neq k; i,  j,  k \in \{1, 2, \cdots n\}}|R_{i}'||R_{j}'||R_{k}'|+\cdots + |R_{1}'||R_{2}'|\cdots|R_{k}'| $,  where each $R_{i}'$ is any set of self-annihilating vertices in a maximal clique in $\Gamma(R_{i})$.
\end{corollary}
	\begin{proof}
		Extending a similar type of construction in the proof of the above theorem,  we can consider $C_1,  C_2,  \cdots C_n$,   a collection of maximal cliques in $\Gamma(R_1), \Gamma(R_2),  \cdots,  \Gamma(R_n)$ respectively. Construct an induced sub graph $X_{i} = \{(0,  0,  \cdots,  c_{i}, \cdots 0)| c_{i}\in C_{i}, \},  X =\bigcup_{i=1}^{n} {X_{i}}$,  where we place the $c_{i}$ in the $i-th$ coordinate. Then $X$ forms a click of cardinality $\sum_{i=1}^{n}cl(\Gamma(R_i))$. Then consider,  $X_{ij} = \{(0,  0,  \cdots,  c_{i},  \cdots c_j \cdots 0)| c_{i}\in R'_{i},  c_{j} \in R'_{j}\}$,  where  $R'_{i},  R'_{j}$ are any set of self annihilating vertices in maximal clique in $\Gamma(R_{i})$ and  $\Gamma(R_{j})$,  where we place the $c_{i}$ and $c_{j}$ in the $i-th$ and $j-th$ entries respectively. Set $Y = \bigcup_{i\neq j; i,  j \in \{1,  2,  \cdots n\}} X_{ij}$. Then $Y$ forms a click of cardinality $\sum_{i\neq j; i,  j \in \{1, 2, \cdots n\}}|R_{i}'||R_{j}'|$,  that is disjoint from $X$. In a similar fashion we can construct $X_{ijk}$ for each distinct triplets $i, j, k \in \{1, 2, \cdots n\}$ and call their union $Z$ and $Z$ gives a click of cardinality $\sum_{i\neq j\neq k;  i, j, k \in \{1, 2, \cdots n\}}|R_{i}'||R_{j}'||R_{k}'|$. Proceeding in this way the result follows.
	\end{proof}

\begin{lemma}\label{3.7} Consider $\Gamma(\mathbb{Z}_n)$ for arbitrary $n$. There is a maximal clique $M$ that contains all self-annihilating vertices.
\end{lemma}
\begin{proof}
Follows from Theorem \ref{2.12} and Lemma \ref{2.10}.
\end{proof}

\begin{theorem}\label{3.8} The clique number of $\Gamma(\mathbb{Z}_n\times\mathbb{Z}_m)$ has a lower bound of $cl(\Gamma(\mathbb{Z}_n)+cl(\Gamma(\mathbb{Z}_m)) + (\frac{n}{n^*}-1)(\frac{m}{m^*}-1)$.
\end{theorem}
\begin{proof}
Follows from Theorem \ref{3.6} and the proof of Theorem \ref{2.12} and Lemma \ref{2.10}.
\end{proof}

\begin{theorem}\label{3.11} $\Gamma(R_1\times\cdots\times R_k)$ where $k \geq 2$ and $R_i$ is a commutative ring with 1 has a simplicial vertex iff some $\Gamma(R_i)$ has a simplicial vertex or some $|R_i| = 2$. 
\end{theorem}
\begin{proof}
Take arbitrary $\Gamma(R_1\times\cdots\times R_k)$. Let some $\Gamma(R_i)$ have a simplicial vertex $c$. Then consider the vertex $(1, \cdots, 1, c, 1, \cdots, 1)$ where $c$ is in the $i$th slot. Each neighbor of $(1, \cdots, 1, c, 1, \cdots, 1)$ must have 0 in every slot except the $i$th slot,  and the value of the $i$th slot must be a neighbor of $c$ in $\Gamma(R_i)$. Since each neighbor of $c$ shares an edge and each other slot is 0,  all such neighbors of $(1, \cdots, 1, c, 1, \cdots, 1)$ form a clique. So $\Gamma(R_1\times\cdots\times R_k)$ has a simplicial vertex.\\
\\
Let some $|R_i| = 2$. Then $(1, \cdots, 1, 0, 1, \cdots, 1)$ only shares an edge with $(0, \cdots, 0, 1, 0, \cdots, 0)$ making $(1, \cdots, 1, 0, 1, \cdots, 1)$ simplicial.\\
\\
Let $\Gamma(R_1\times\cdots\times R_k)$ have a simplicial vertex $v$. Also, 
assume all $|R_i| > 2$ and no $\Gamma(R_i)$ have any simplicial vertices. Consider arbitrary $v$ in $\Gamma(R_1\times\cdots\times R_k)$. Let $v$ have 0 at some index,  $v = (\cdots, 0, \cdots)$. Then since no $|R_i| = 2$,  there exists some vertex $a\in R_i$ that is not 0 or 1. $v$ then shares an edge with $(0, \cdots, 0, 1, 0, \cdots, 0)$ and $(0, \cdots, 0, a, 0, \cdots, 0)$ which do not share an edge. So for $v$ to be simplicial,  it cannot contain any 0. Let $v$ have $a$ at some index,  where $a$ is a zero divisor in its respective $\Gamma(R_i)$. $v = (\cdots, a, \cdots)$. Then $v$ shares an edge with every $(0, \cdots, 0, a', 0, \cdots, 0)$ where $a\cdot a' = 0$ in $\Gamma(R_i)$. $a$ is not simplicial since no $\Gamma(R_i)$ have any simplicial vertex,  so some neighbor $(0, \cdots, 0, a', 0, \cdots, 0)$ will not share an edge with another neighbor of the same form. So $v$ is not simplicial if it has any zero-divisors in its slots. For $v$ to be simplicial,  every slot must be a non-zero,  non-zero-divisor. However,  elements of that form are not vertices. So $\Gamma(R_1\times\cdots\times R_k)$ has no simplicial vertices,  which is a contradiction. The assumption that all $|R_i| > 2$ and no $\Gamma(R_i)$ have any simplicial vertices is false. So some $|R_i| > 2$ or some $\Gamma(R_i)$ has a simplicial vertex.
\end{proof}

\begin{theorem}\label{3.12} $\Gamma(R_1\times\cdots\times R_k)$ where $R_i$ is a commutative ring with 1 is non-chordal if any $\Gamma(R_i)$ is non-chordal.
\end{theorem}
\begin{proof}
Consider arbitrary $\Gamma(R_1\times\cdots\times R_k)$. Then let some $\Gamma(R_i)$ be non-chordal. So there exists a cycle $a_1-a_2-\cdots-a_k-a_1$ greater than 3 with no chords. Then in $\Gamma(R_1\times\cdots\times R_k)$,  there is a cycle $(0, .., a_1, \cdots, 0)-(0, \cdots, a_2, \cdots, 0)-\cdots-(0, \cdots, a_k, \cdots, 0)-(0, \cdots, a_1, \cdots, 0)$,  which makes it non-chordal.
\end{proof}

\begin{lemma}\label{3.13} $\Gamma(R_1\times\cdots\times R_k)$ where $R_i$ is a commutative ring with 1 and $k\geq 2$ is non-chordal if more than one $|R_i| \geq 3$.
\end{lemma}
\begin{proof}
In $\Gamma(R_1\times\cdots\times R_k)$,  let two or more $|R_i| \geq 3$. Without loss of generality,  let the first two slots be the $R_i$ with a magnitude greater than or equal to 3. Then $(1,  0, \cdots, 0)-(0,  1, \cdots, 0)-(a,  0, \cdots, 0)-(0,  b, \cdots, 0)$ where $a$ is a non-trivial element of $R_1$ and $b$ is a non-trivial element of $R_2$,  is a cycle with no chord. So $\Gamma(R_1\times\cdots\times R_k)$ is non-chordal.
\end{proof}

\begin{lemma}\label{3.14} $\Gamma(R_1\times\cdots\times R_k)$ where $R_i$ is a commutative ring with 1 is non-chordal if $k \geq 4$.
\end{lemma}
\begin{proof}
Let $k>4$. Then $(1, 1, 0, 0, \cdots, 0)-(0, 0, 1, 1, \cdots, 0)-(1, 0, 0, 0, \cdots, 0)-(0, 0, 0, 1, \cdots, 0)$ is a chord-less cycle. So $\Gamma(R_1\times\cdots\times R_k)$ is non-chordal.
\end{proof}

\begin{lemma}\label{3.15} $\Gamma(\mathbb{Z}_{n_1}\times\mathbb{Z}_{n_2}\times\mathbb{Z}_{n_3})$ where at least one $n_i>2$ is non-chordal.

\end{lemma}
\begin{proof}
Without loss of generality,  let $n_3 > 2$. Then, \\
$(1, 0, 0)-(0, 0, 2)-(1, 1, 0)-(0, 0, 1)$ is a chord-less cycle.
\end{proof}

\begin{theorem}\label{3.16} The only chordal $\Gamma(\mathbb{Z}_{n_1}\times \mathbb{Z}_{n_2}\times \cdots\times\mathbb{Z}_{n_k})$ where $n_i \geq 2$ and $k\geq 2$ are $\Gamma(\mathbb{Z}_2\times\mathbb{Z}_p)$,  $\Gamma(\mathbb{Z}_2\times\mathbb{Z}_{p^2})$ and $\Gamma(\mathbb{Z}_2\times\mathbb{Z}_2\times\mathbb{Z}_2)$.
\end{theorem}
\begin{proof}

Consider $\Gamma(\mathbb{Z}_2\times\mathbb{Z}_p)$. Since $\Gamma(\mathbb{Z}_p)$ has no vertices,  the only vertices of $\Gamma(\mathbb{Z}_2\times\mathbb{Z}_p)$ are $(1,  0)$ or of the form $(0,  x)$ where $0<x<p$. So the graph is a star making it chordal. \\

Consider $\Gamma(\mathbb{Z}_2\times\mathbb{Z}_{p^2})$. Assume that $\Gamma(\mathbb{Z}_2\times\mathbb{Z}_{p^2})$ is non-chordal. Then there exists a cycle $C$ length greater than $3$ that has no chord. Let $v$ be an arbitrary vertex in $C$.\\
Let $v$ have a multiple of $p$ as its second element,  $v=(a, bp)$. Then every vertex that is not a neighbor of $v$ in $C$ must have a non-zero non-multiple of $p$ as its second element. Therefore,  both neighbors of $v$ must have 0 as their second element so that they share an edge with their other neighbor. So both neighbors of $v$ are $(1, 0)$. We cannot repeat vertices so $v$ cannot have a multiple of $p$ as its second element. That means the only possible vertices in $C$ are $(1, 0)$ and $(0, b)$ where $b$ is a non-zero non-multiple of $p$. A cycle of size $4$ or greater cannot be constructed out of these vertices since we cannot write $(1, 0)$ twice and $(0, b)$ does not share an edge with itself. $C$ cannot be constructed,  so $\Gamma(\mathbb{Z}_2\times\mathbb{Z}_{p^2})$ is chordal. \\

Consider $\Gamma(\mathbb{Z}_2\times\mathbb{Z}_2\times\mathbb{Z}_2)$. The graph of $\Gamma(\mathbb{Z}_2\times\mathbb{Z}_2\times\mathbb{Z}_2)$ is shown below and is chordal. \\

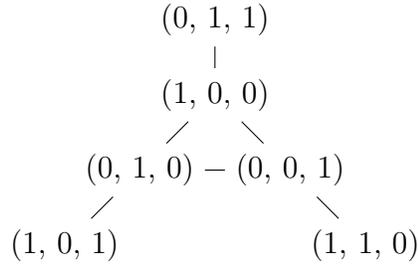
\begin{tikzpicture}

\node (a) at (1, 1) {(1, 0, 1)};
\node (b) at (2, 2) {(0, 1, 0)};
\node (c) at (4, 2) {(0, 0, 1)};
\node (d) at (3, 3) {(1, 0, 0)};
\node (e) at (3, 4) {(0, 1, 1)};
\node (f) at (5, 1) {(1, 1, 0)};

\foreach \from/\to in {b/c, c/d, b/d, a/b, c/f, d/e}
 \draw (\from) -- (\to);

\end{tikzpicture}\\

To prove the converse,  let's assume the opposite. Let there be a chordal $\Gamma(\mathbb{Z}_{n_1}\times \mathbb{Z}_{n_2}\times \cdots\times\mathbb{Z}_{n_k})$ not listed. By Lemma \ref{3.13},  only one $n_i$ can be greater than 2. By Lemma \ref{3.14},  $k \leq 3$. By Theorem \ref{3.12},  if any $n_i$ are non-chordal,  $\Gamma(\mathbb{Z}_{n_1}\times \mathbb{Z}_{n_2}\times \cdots\times\mathbb{Z}_{n_k})$ will be non-chordal. So every $n_i$ must be $p^x$,  $2p$,  or $2p^2$ which was shown by Theorem \ref{3.15}. 

So the only possible $\Gamma(\mathbb{Z}_{n_1}\times \mathbb{Z}_{n_2}\times \cdots\times\mathbb{Z}_{n_k})$ are $\Gamma(\mathbb{Z}_2\times\mathbb{Z}_{p^x})$,  $\Gamma(\mathbb{Z}_2\times\mathbb{Z}_{2p})$,  $\Gamma(\mathbb{Z}_2\times\mathbb{Z}_{2p^2})$,  $\Gamma(\mathbb{Z}_2\times\mathbb{Z}_2\times\mathbb{Z}_{p^x})$,  $\Gamma(\mathbb{Z}_2\times\mathbb{Z}_2\times\mathbb{Z}_{2p})$ and $\Gamma(\mathbb{Z}_2\times\mathbb{Z}_2\times\mathbb{Z}_{2p^2})$.

In $\Gamma(\mathbb{Z}_2\times\mathbb{Z}_{p^x})$ where $x \geq 3$ and $p$ is prime,  $(1, p^{x-1})-(0, (p-1)p)-(1, 0)-(0, p)$ is a chord-less cycle.\\

In $\Gamma(\mathbb{Z}_2\times\mathbb{Z}_{2p})$ where $p \geq 3$ is a prime,  $(1, 0)-(0, 4)-(1, p)-(0, 2)$ is a chord-less cycle.\\

In $\Gamma(\mathbb{Z}_2\times\mathbb{Z}_{2p^2})$ where $p \geq 3$ is a prime,  $(1, 2p)-(0, p)-(1, 4p)-(0, p^2)$ is a chord-less cycle.\\

By Lemma 3.15,  $\Gamma(\mathbb{Z}_2\times\mathbb{Z}_2\times\mathbb{Z}_{p^x})$,  $\Gamma(\mathbb{Z}_2\times\mathbb{Z}_2\times\mathbb{Z}_{2p})$ and $\Gamma(\mathbb{Z}_2\times\mathbb{Z}_2\times\mathbb{Z}_{2p^2})$ are all non-chordal where $p \geq 3$.\\

So there are no other chordal $\Gamma(\mathbb{Z}_{n_1}\times \mathbb{Z}_{n_2}\times \cdots\times\mathbb{Z}_{n_k})$.
\end{proof}

\begin{lemma}\label{3.17} $D(\Gamma(\mathbb{Z}_{n_1}\times \mathbb{Z}_{n_2}\times \cdots\times\mathbb{Z}_{n_k}))$ has an upper bound of $2[D(\Gamma(\mathbb{Z}_{n_1})) + D(\Gamma(\mathbb{Z}_{n_2})) + \cdots + D(\Gamma(\mathbb{Z}_{n_k}))]$.
\end{lemma}
\begin{proof}
Let $d_i$ be the domination number of $\Gamma(\mathbb{Z}_{n_i})$. Then each $\Gamma(\mathbb{Z}_{n_i})$ has a dominating set $D_i$ size $d_i$ Then consider the sets\\
$A_i = \{ (0, \cdots, 0, v, 0, \cdots, 0) \mid| v\in D_i\}$ where the $i$-th slot is filled with arbitrary vertex in $D_i$ and the rest are 0. Also consider $B_i$,  the set of neighbors of each vertex in $A_i$,  with only one neighbor for each vertex. Now consider $\cup_{i=1}^k (A_i\cup B_i)$ and arbitrary vertex $v \in \Gamma(\mathbb{Z}_{n_1}\times \mathbb{Z}_{n_2}\times \cdots\times\mathbb{Z}_{n_k})$.\\
\begin{enumerate}
\item[Case 1:] $v$ has an element $e$ in some $i$th slot that is a vertex of $\Gamma(\mathbb{Z}_{n_i})$.\\
If $e\in D_i$,  then $v$ shares an edge with the corresponding vertex in $B_i$,  and if $e\notin D_i$,  then $v$ shares an edge with some vertex in $D_i$.\\
\item[Case 2:] $v$ has 0 in some $i$th slot.\\

Then $v$ shares an edge with some vertex in $A_i$.\\
So if neither of these cases is true,  none of the elements of $v$ can be zero or a vertex in its corresponding $\Gamma(\mathbb{Z}_{n_i})$,  so the only neighbor of $v$ is $(0, .., 0)$ which means $v$ is not a vertex. Then $\cup_{i=1}^k (A_i\cup B_i)$ is a dominating set. $A_i$ and $B_i$ both have size $D(\Gamma(\mathbb{Z}_{n_i}))$ since it only has one vertex for each vertex in its corresponding $D_i$. So the size of $\cup_{i=1}^k (A_i\cup B_i)$ is  $2(D(\Gamma(\mathbb{Z}_{n_1})) + D(\Gamma(\mathbb{Z}_{n_2})) + \cdots + D(\Gamma(\mathbb{Z}_{n_k})))$ which is an upper bound of the domination number.
\end{enumerate}
\end{proof}

\begin{lemma}\label{3.18} $D(\Gamma(\mathbb{Z}_{n_1}\times \mathbb{Z}_{n_2}\times \cdots\times\mathbb{Z}_{n_k}))$ has a lower bound of $D(\Gamma(\mathbb{Z}_{n_1}))+D(\Gamma(\mathbb{Z}_{n_2}))\\
+\cdots+D(\Gamma(\mathbb{Z}_{n_k}))$.
\end{lemma}
\begin{proof}
Let $D$ be an arbitrary dominating set of $\Gamma(\mathbb{Z}_{n_1}\times \mathbb{Z}_{n_2}\times \cdots\times\mathbb{Z}_{n_k})$. Consider the vertex $v=(1, \cdots, 1, a, 1, \cdots, 1)$ where $a$ in the $i$th slot is a vertex of $\Gamma(\mathbb{Z}_{n_i})$. The only possible neighbors of $v$ are of the form $(0, \cdots, 0, b, 0, \cdots, 0)$ where $b$ is a neighbor of $a$ in $\Gamma(\mathbb{Z}_{n_i})$. Construct a subset $A_i$ that is all vertices in $D$ of the form $(0, .., 0, b, 0, \cdots, 0)$ or $(1, \cdots, 1, b, 1, \cdots, 1)$ where $b$ is a vertex in $\Gamma(\mathbb{Z}_{n_i})$.\\
We claim that arbitrary $A_i$ has a size of at least $d_i$,  where $d_i$ is the domination number of $\Gamma(\mathbb{Z}_{n_i})$.\\
Assume otherwise. Then there are less than $d_i$ vertices of the form $(0, \cdots, 0, b, 0, \cdots, 0)$ and $(1, \cdots, 1, b, 1, \cdots, 1)$. Take some $a$ in $\Gamma(\mathbb{Z}_{n_i})$. Since some vertex in $D$ shares an edge with every vertex not in $D$,  arbitrary $v=(1, \cdots, 1, a, 1, \cdots, 1)$ either shares an edge with some $(0, \cdots, 0, b, 0, \cdots, 0)$ or is itself in $D$ and therefore in $A_i$. If $v$ is not in $D$,  then $v$ shares an edge with some $(0, \cdots, 0, b, 0, \cdots, 0)$ which means $a$ shares an edge with some $b$ in $\Gamma(\mathbb{Z}_{n_i})$. If $v$ is in $D$,  then $(1, \cdots, 1, a, 1, \cdots, 1)$ is in $A_i$. Construct a set $H$ that contains all $a$ in the $i$th slot of all such $v$. $H$ forms a dominating set of $\Gamma(\mathbb{Z}_{n_i})$ size less than $d_i$ which is a contradiction since $D(\Gamma(\mathbb{Z}_{n_i}) = d_i$. So $A_i$ has a size of at least $d_i$.\\
Next,  each $A_i$ is disjoint from each other since $b$ in the $i$th slot can never be 0 or 1,  which means there will be no duplicate vertices. So the sum of the sizes of each $A_i\subseteq D$ will be greater than or equal to the sum of the domination number of each $\Gamma(\mathbb{Z}_{n_i})$. This is a lower bound of the domination number. 
\end{proof}

Combining Lemma \ref{3.17} and \ref{3.18} we get the following.

\begin{theorem}\label{3.19} $D(\Gamma(\mathbb{Z}_{n_1}))+D(\Gamma(\mathbb{Z}_{n_2}))+\cdots+D(\Gamma(\mathbb{Z}_{n_k}))$ $\leq$ $D(\Gamma(\mathbb{Z}_{n_1}\times \mathbb{Z}_{n_2}\times \cdots\times\mathbb{Z}_{n_k}))$ $\leq$ $2[D(\Gamma(\mathbb{Z}_{n_1})) + D(\Gamma(\mathbb{Z}_{n_2})) + \cdots + D(\Gamma(\mathbb{Z}_{n_k}))]$.\\
\end{theorem}

The next theorem talks about the coefficients of a Domination Polynomial.\\

\begin{theorem}\label{3.20} In arbitrary $\Gamma(\mathbb{Z}_{p_1^{\alpha_1}p_2^{\alpha_2}\cdots p_k^{\alpha_k}})$,  $k \geq 3$,  and each $p_i$ is a distinct prime number,  the coefficient $c$ of the smallest degree of the domination polynomial is $(p_1-1)(p_2-1)\cdots(p_k-1)$.
\end{theorem}
	\begin{proof}
		Consider $\Gamma(\mathbb{Z}_{p_1^{\alpha_1}p_2^{\alpha_2}\cdots p_k^{\alpha_k}})$,  $k \geq 3$. Let $n = p_1^{\alpha_1}p_2^{\alpha_2}\cdots p_k^{\alpha_k}$. Construct set $D$ that has exactly one element from each type class $T_{n/p_i}$. Since every vertex of $\Gamma(\mathbb{Z}_{p_1^{\alpha_1}p_2^{\alpha_2}\cdots p_k^{\alpha_k}})$ must be a multiple of some $p_i$,  every vertex shares an edge with some vertex in some $T_{n/p_i}$ and therefore,  $D$. So $D$ is a dominating set.\\
		We claim that for arbitrary minimal dominating set $D$,  exactly one vertex must be present from each type class $T_{n/p_i}$.\\
		Assume the opposite. Then there exists a dominating set $D$ that either doesn't have a vertex from some type class $T_{n/p_x}$ or has an extra vertex not in any type class $T_{n/p_x}$. Let $D$ not have any vertices from $T_{n/p_x}$. Since the only neighbors of vertices in $T_{p_x}$ are in $T_{n/p_x}$,  every vertex in $T_{p_x}$ must be in $D$. $p_x \neq n/2$ because otherwise $2p_x=n$,  $2p_x=p_1^{\alpha_1}p_2^{\alpha_2}\cdots p_k^{\alpha_k}$ and $2 = \frac{p_1^{\alpha_1}p_2^{\alpha_2}\cdots p_k^{\alpha_k}}{p_x}$ which is not possible. So by Lemma 1.10,  $T_{p_x}$ has more than one element. If $D$ contains at least one vertex from all $T_{n/p_x}$ except $T_{p_x}$,  then the size of $D$ is larger than $A$ above. So $D$ is not a minimal dominating set. If $D$ lacks any vertices from other $T_{n/p_i}$,  then for each vertex missing,  two or more must be added from $T_{p_x}$. In which case $D$ would also not be minimal. So $D$ must contain at least one vertex from each $T_{n/p_i}$. There also cannot be any additional vertices,  either from type classes not of the form $T_{n/p_x}$,  nor multiple from the same type class. Otherwise $D$ would not be minimal.\\
		Since there are $p_i - 1$ vertices in $T_{n/p_i}$,  the total amount of possible minimal dominating sets $D$ is $(p_1-1)(p_2-1)\cdots(p_k-1)$.
	\end{proof}

\begin{theorem}\label{3.21} $\Gamma(\mathbb{Z}_{p_1}\times\mathbb{Z}_{p_2}\times\cdots\times\mathbb{Z}_{p_k})$ is k-partite where every $p_i$ is prime.
\end{theorem}
\begin{proof}
Consider some graph $\Gamma(\mathbb{Z}_{p_1}\times\mathbb{Z}_{p_2}\times\cdots\times\mathbb{Z}_{p_k})$. Construct a collection of subsets $S_i$ which is the set of all vertices with a non-zero term in the $i$th slot and zero in any slot less than $i$.\\
$S_1 = \{ (a, \cdots.) | a\in \mathbb{Z}_{p_1},  a\neq 0 \}$\\
$S_2 = \{ (0,  a,  \cdots) | a\in \mathbb{Z}_{p_2},  a\neq 0 \}$\\
$\cdots$\\
$S_k = \{ (0,  0,  \cdots,  0,  a) | a\in \mathbb{Z}_{p_k},  a\neq 0 \}$\\
We claim that no two vertices $u, v$ from the same subset $S_x$ share an edge.\\
Consider arbitrary vertices $u$ and $v$ in some $S_x$. By the definition,  the $x$th slot of $u$ and $v$ has a non-zero term from $\mathbb{Z}_{p_x}$. Since $\mathbb{Z}_{p_x}$ has no non-zero,  zero divisors,  the terms in the $x$th slot will not multiply to get 0,  so $u$ and $v$ do not share an edge.\\
No we claim that all the $S_i$ form a partition of $\Gamma(\mathbb{Z}_{p_1}\times\mathbb{Z}_{p_2}\times\cdots\times\mathbb{Z}_{p_k})$.\\
Assume there is a vertex $v$ not in any $S_i$. Let $v$ have a non-zero element $a$ in the $x$th slot. The by definition it is in $S_x$,  or in some $S_i$,  $i<x$ if some $i$th slot also has a non-zero element. So $v$ cannot have any non-zero elements and thus,  $v=0$ which is not a vertex. $\cup S_i$ is the vertex set of  $\Gamma(\mathbb{Z}_{p_1}\times\mathbb{Z}_{p_2}\times\cdots\times\mathbb{Z}_{p_k})$.\\
Assume there is a $v$ in multiple $S_i$,  say $S_x$ and $S_y$. Without loss of generality,  let $x<y$. Then the $x$th slot of $v$ has a non-zero term $a$ because it is in $S_x$. But the $x$th slot must be zero because $v$ is in $S_y$. That is a contradiction,  so there are no overlaps in the partition.\\
$S_i$ is a k-partite partition,  so $\Gamma(\mathbb{Z}_{p_1}\times\mathbb{Z}_{p_2}\times\cdots\times\mathbb{Z}_{p_k})$ is k-partite.
\end{proof}

\section{Zero divisor graph of the poset $D_n$}

Zero divisor graph of a poset has been studied in ~\cite{jW99}, ~\cite{jW100}, ~\cite{jW101}. We always have Clique number of the zero divisor graph of a ring does not exceed the Chromatic number of that. Beck conjectured that that for an arbitrary ring $R$,  they are same. But Anderson and Naseer~\cite{jW102} have shown that this is not the case in general,  namely they presented an example of a commutative local ring $R$ with $32$ elements for which Chromatic number is strictly bigger than the clique number.In ~\cite{jW102} Nimbhorkar,  Wasadikar and DeMeyer have shown that Beck's conjecture holds for meet-semilattices with $0$, i.e., commutative semigroups with $0$ in which each element is idempotent. Infact, it is valid for a much wider class of relational structures, namely for partially ordered sets (posets, briefly) with $0$. Now, to any poset $(P, \leq)$,  with a least element $0$ we can assign the graph $G$ as follows: its vertices are the nonzero zero divisors of $P$,  where a nonzero $x \in P$ is a called a zero divisor if there exists a non zero $y\in P$,  so that $L(x, y)=0,  L(x, y)=\{z\in P|z\leq x, y\} $. And  $x, y$ are connected by an edge if $L(x, y)=0$.
We discuss here some properties of the zero divisor graph of a specific poset $D_n$. Very often we used the prime factorization of the positive integer $n$. By abuse of notation,  let us call $D_n$ as the zero divisor graph of the poset $D_n$. Note that,  the vertex set of $D_n$ is the set of all factors of $n$ that are not divisible by some prime factor of $n$. Also,  note that,  two vertices in $D_n$ are connected by an edge if and only if they are mutually co-prime.\\

\begin{remark}[Properties of $D_n$]
$\phantom e$\\
\begin{enumerate}
	\item[i.] If $n =p^{m}$ for some prime $p$ and positive integer $m$,  then $D_{n}$ is trivial.\\

	So from now on consider $D_n$ where $n \neq p^{m}$ where $p$ and $m$ are as mentioned.\\
	\item [ii.] The diameter of $D_n$ is 3 iff $n$ has three distinct prime factors namely $p$,  $q$,  $r$. This is shown by the path $pq - r - p - qr$. Otherwise,  the diameter is 1 or 2,  as $D_{p^mq^n}$ is complete bipartite which has diameter 2 or in the case of $m = n = 1$ has diameter 1.  \cite{jW103} shows zero divisors of a poset have diameter of 1,  2,  or 3.
	\item[iii.] $D_n$ is complete only when $n=pq$,  where $p$ and $q$ are two distinct primes. $D_n$ is complete bipartite iff $n = p^{m}q^{s}$ where $m$ and $s$ are two positive integers.\\

	\item[iv.]We have the clique number of $D_n$ and a few coefficients of the clique polynomial.The clique number of $D_n$ is the number of distinct prime factors of $n$. If $n= p_{1}^{\alpha_{1}}p_{2}^{\alpha_{2}}p_{3}^{\alpha_{3}}\cdots p_{r}^{\alpha_{r}}$ where $p_{i}$'s are  distinct primes $\forall i$,  any set of vertices $\{ p_{1}^{\beta_{1}},  p_{2}^{\beta_{2}},  p_{3}^{\beta_{3}}\cdots p_{r}^{\beta_{r}}\}$,  where $1\leq \beta_{i}\leq \alpha_{i}$ $\forall i$ forms a maximal clique. This is a clique because no two vertices in a clique can have a common prime factor. Also,  if any vertex of a clique has more than one prime factor,  the clique will not be maximal. Hence the clique number is $r$,  the number of distinct primes of $n$. And the leading coefficient in the clique polynomial is $\alpha_{1}\alpha_{2}\cdots\alpha_{r}$. The coefficient of $x^{r-1}$ is $\sum_{i=1}^{r}(\alpha_1\alpha_2\cdots \alpha_{i-1}\alpha_{i+1}\cdots \alpha_{r})+\binom {r}{2}\alpha_1\alpha_2\cdots\alpha_r.$ Reason: Consider a clique of size $r-1$. If all the vertices has single prime factors then,  there are $\sum_{i=1}^{r}(\alpha_1\alpha_2\cdots \alpha_{i-1}\alpha_{i+1}\cdots \alpha_{r})$ many of this type,  as a typical clique of this type is a set of the form $\{p_{1}^{\beta_{1}}, p_{2}^{\beta_{2}}, \cdots p_{i-1}^{\beta_{i-1}}, p_{i+1}^{\beta_{i+1}}, \cdots p_{r}^{\beta_{r}} \}$,  where $1\leq \beta_{j}\leq \alpha_{j}\forall j \in \{1, 2, \cdots r\}$. Otherwise,  exactly one vertex will contain two primes. And in that case we will obtain $\binom {r}{2}\alpha_1\alpha_2\cdots\alpha_r$ many such clique sets with cardinality $r-1$. No element in a clique set can have three distinct primes in it's prime factorization. Hence the result follows.\\

	\item[v.]The domination number of $D_n$ is the number of distinct prime factors of $n$,  same as the clique number,  as any dominating set must not omit a prime factor of $n$. If some $p_{i}$ is missing from a set of vertices $V$,  then the vertex $p_{1}p_{2}\cdots p_{i-1}p_{i+1}\cdots p_{r}$ is not adjacent to any vertex in $V$. Furthermore,  if we let $V$ be the set of all distinct primes of $n$,  each vertex in $D_n$ must share an edge with at least one vertex in $V$ because each vertex in $D_n$ must omit at least one prime of $n$ from its prime factorization.\\

	\item[vi.]$D_{n}$ is regular iff $n= (pq)^{m}$ for some positive integer $m$. If $n = p^{m}q^{r},  m\neq r$,  then $D_{n}= K_{m, n}$ which is not regular. Then,  if $n$ has more than two distinct primes in it's prime factorization,  $p$ and $pq$ are vertices with a different degrees. Every vertex that shares an edge with $pq$ shares an edge with $p$,  but $p$ shares an edge with $q$ while $pq$ does not,  making the graph non-regular.\\

	\item[vii.]In \cite{jW100},  it is discussed that the girth of the zero divisor graph of any poset is 3, 4,  or $\infty$. The girth of $D_{n}$ is $\infty$ iff $n= p^{m}q$,  where $p$ and $q$ are two distinct primes and $m$ is a positive integers bigger than $1$. The girth of $D_{n}$ is $4$,  if and only if $n= p^{m}q^{r}$,  where $p$ and $q$ are two distinct primes and $m$ and $r$ are both positive integers bigger than $1$. Otherwise,  the girth of $D_{n}$ is $3$,  because if $n$ has at least $3$ different prime factors $p$,  $q$ and $r$,  then $p-q-r-p$ is a triangle in $D_{n}$.\\

	\item[viii.] $D_n$ is not perfect if $n$ is the product of least five different primes $p, q, r, s, t$ in it's prime factorization,  then $ps-qt-pr-qs-tr-ps$ is a cycle of length five in $D_n$. Hence by Strong perfect graph theorem $D_n$ is not perfect.\\
		Suppose $n$ has 4 distinct prime factors $p$,  $q$,  $r$ and $s$. Assume there is an odd cycle of length 5 or greater that contains a vertex $v$ that is the product of two such primes.  Let $v = p^xq^y$. Then the two neighbors of $v$ cannot be a multiple of $p$ or $q$. Suppose the neighbors both consists of  $r^a$ for some positive integer $a$. Then,  we get part of the cycle as $r^a - p^xq^y - r^b$ for another positive integer $b$. Then,  $r^a$ will necessarily share an edge with the other neighbors of $r^b$ making the cycle length 4. So the neighbors of $v$ must have both $r$ and $s$. Additionally,  these part of the cycle must be of the form $r^a - p^xq^y - r^ws^z$,  otherwise we get a cycle of length $4$ again. But any vertex that shares an edge with $r^ws^z$ must also share an edge with $r^a$ making such a cycle impossible. This means any odd cycle length greater than 5 cannot contain a vertex with two or more prime factors,  making an odd cycle length greater than 4 impossible.The other two situations when $v$ consists of only one prime,  or three primes also gives contradiction.Thus $D$ is perfect iff $n$ has 4 or fewer prime factors (the $n$ $< 4$  case follows).\\
	
	\item[ix.] $D_n$ is chordal iff $n = p^mq$ or $n = pqr$ where $p$,  $q$ and $r$ are distinct primes and $m \geq 1$. For if $n$ is not of that form,  $p-q-p^{2}-q^{2}-p$ or $p - q - p^2 - qr - p$ or $p - r - pq - rs - p$ will give holes of length greater than $3$ in respective $D_n$'s.\\

	\item[x.]Let,  $n$ be a square free positive integer. Then,  it's simplicial vertices are precisely those factors of $n$ which misses exactly one prime in it's prime factorization. Now,  suppose $n$ is not square free. Then,  if all primes in it's prime factorization are not square free,  it has no simplicial vertex. Otherwise,  the simplicial vertices are precisely those which misses exactly one square free prime factor. For example,  if $n = p^2q^2r$,  $pq$,  $p^2q$,  $pq^2$ and $p^2q^2$ are the only simplicial vertices because $r$ is the only square free prime factor.\\

	\item[xi.] The only planar $D_n$ has $n$ of the form $n = p^mq$,  $p^mq^2$,  $pqr$ or $p^2qr$. First,  let $n$ have only 2 prime factors. We will first examine this case. If $n = p^mq^l$ where $l \geq 3$ and $m \geq 3$,  then $K_{3, 3}$ is a subgraph of $D_n$ and therefore a minor of $D_n$. Then by Wagner's theorem,  $D_n$ is non planar. But in the case of $p^mq$,  $D_n$ is a star,  so it is planar. And in the case of $p^mq^2$,  the graph can be drawn without any crossing edges. Next,  let's examine $n$ with 3 prime factors. If $n = pqr$ or $n = p^2qr$ the graph is clearly planar if drawn. But,  if $n = p^mqr$ where $m \geq 3$,  The subgraph consisting of $p$,  $p^2$,  $p^3$,  $q$,  $r$ and $qr$ form $K_{3, 3}$ if we delete the edge between $q$ and $r$. Then by Wagner's theorem the graph is non-planar since $K_{3, 3}$ is a minor. Next,  if $n = p^mq^lr$,  where $m \geq 2$ and $l \geq 2$ the set of vertices $q$,  $q^2$,  $p$,  $p^2$,  $r$,  $pr$ and $qr$ is a subdivision of $K_5$. Then,  by Kuratowski's theorem,  the graph is non-planar. So the only planar $D_n$ with only 3 primes in $n$ are $pqr$ and $p^2qr$. Lastly,  consider the case where $n$ has 4 primes in its prime factorization,  $n = pqrs$. Then,  the vertex set of $p$,  $q$,  $r$,  $s$,  $pq$ and $rs$ can be made isomorphic to $K_5$ by contracting the edge between $pq$ and $rs$ to make a single vertex. Therefore,  $K_5$ is a minor of $D_n$ for this case,  and by Wagner's theorem the graph is non-planar.\\

	\item[xii.] $D_n$ is Eulerian iff the power of each prime in the prime factorization of $n$ is even.
	For,  if $n$ has a prime $p^{\alpha}$ that appears in it's prime factorization where $\alpha$ is odd,  then the vertex $\frac{n}{p^\alpha}$ has odd degree,  otherwise every vertex has even degree.\\

	\item[xiii.] If $n$ is square free,  then we have the edge cardinality of $D_n$ as $\sum_{i=1}^{r-1}2^{r-i-1}\binom{r}{i}-2^{r-1}-1$,  where $r$ is the number of distinct primes of $n$.
	For,  if we consider $n= p_{1}p_{2}\cdots p_{r}$,  where $p_{i}$'s are distinct primes,  then the degree of each vertex $p_{i}$ is $\sum_{i=1}^{r-1}\binom{r-1}{i}= 2^{r-1}-1$ giving $r(2^{r-1}-1)$ to the degree sum of the vertices. Similarly each vertex $p_{i}p_{j}$ is adjacent to $\sum_{i=1}^{r-2}\binom{r-2}{i} =2^{r-2}-1$
	many vertices,  giving $\binom{r}{2}(2^{r-1}-2)$ in the degree sum. Proceeding in this way,  we obtain the sum of the vertex degrees are $\sum_{i=1}^{r-1}\binom{r}{i}(2^{r-i}-1) = \sum_{i=1}^{r-1}\binom{r}{i}2^{r-i}-2^{r}-2$. Then,  as the sum of vertex degrees is twice the edge cardinalities the result follows.\\

	\item[xiv.] We have a lower bound for Independence number of $D_n$. Let,  
	$n= p_{1}^{\alpha_{1}}p_{2}^{\alpha_{2}} \cdots p_{r}^{\alpha_{r}}$,  where $p_{i}$'s are distinct primes.Then if $I$ is the independence number of $D_n$, \\$ I\geq Max_{1\leq i\leq r}[ \alpha_{i}\{1+\sum \alpha_{i_1}\alpha_{i_2}\cdots \alpha_{i_l} |\alpha_{i}\neq \alpha_{i_j}\neq \alpha_{i_k},  j, k \in\{1, 2, \cdots l\}\}]$,  $\{\alpha_{i_1}, \alpha_{i_2}\cdots \alpha_{i_l}\}$ varies over all non empty proper subset of $\{\alpha_1, \alpha_2\cdots \alpha_{i-1},  \alpha_{i+1}\cdots \alpha_r\}$\\
	\begin{proof} Consider any independent set containing  $p_{i}$ from the list of primes in the prime factorization of $n$. Then,  the largest possible independent set containing $p_{i}$,  will have $p_{i}$ as a factor in all it's vertices. So,  that must contain $p_i,  p_{i}^{2},  \cdots,  p_{i}^{\alpha_{i}}$ giving $\alpha_{i}$ many vertices in the independent set. In order to maximize the cardinality of the set,  we need to consider all possible factors of $n$ that has a factor $p_{i}$ and that misses atleast one prime in the prime factorization of $n$. Thus we get\\
$p_{i}p_{1}, p_{i}p_{1}^{2},  \cdots p_{i}p_{1}^{\alpha_{1}}, p_{i}p_{2}, p_{i}p_{2}^{2},  \cdots p_{i}p_{2}^{\alpha_{2}},  \cdots p_{i}p_{r}, p_{i}p_{r}^{2},  \cdots p_{i}p_{r}^{\alpha_{r}}$ are inside the independent set giving $\alpha_{i}(\alpha_{1}+\alpha_{2}+\cdots \alpha_{i-1}+\alpha_{i+1}+\cdots \alpha_{r})=\alpha_{i}\sum_{j=1, j\neq i}^{r}\alpha_{j}$ many vertices. Similarly,  we get more $\alpha_{i}\sum_{j=1, i\neq j\neq k}^{r}\alpha_{j}\alpha_{k}$ many vertices from the factors of $n$ that contains $p_{i}$ and that are product of three primes. Proceeding in this way get the necessary result. 
\end{proof}
	\item [xv.] Let,  $n$ be square free. Then,  a lower bound of the Independence number of $D_n$ is $ 2^{r-1}-r$,  where $r$ is the number of prime factors of $n$. If,  
	$n= p_{1}p_{2} \cdots p_{r}$,  where $p_{i}$'s are distinct primes,  then whenever $I$ is the independence number of $D_n$,  $I \geq 2^{r-1}-r$.
	\begin{proof} Consider any independent set in $D_n$. If we pick up any element from that set,  that is divisible by some $p_{i}$,  then,  all possible proper divisors of $n$,  that has $p_{i}$ as a factor forms an Independent set of $D_n$ and cardinality of that set is  $2^{r-1}-r$. Hence the result follows.
\end{proof}

\end{enumerate}

\end{remark}

\section{acknowledgment}
The authors acknowledge Dr. Lisa DeMeyer for introducing this topic to us by an excellent presentation, which motivated us to work in this area.

\bibliographystyle{amsplain}

\end{document}